\documentclass[12pt,reqno]{amsart}

\usepackage{amssymb, amsmath, amsfonts, latexsym}
\usepackage{enumerate}

\usepackage{color}
\usepackage{diagbox}
\newtheorem{thm}{Theorem}[]
\newtheorem{lem}{Lemma}[section]

\newtheorem{rmk}{Remark}[section]
\theoremstyle{definition}

\usepackage{graphicx}
\usepackage{subcaption}
\numberwithin{equation}{section} \theoremstyle{remark}

\title[Product of Ginibre]{\bf Phase transition and precise convergence rate of  spectral radius of product of complex Ginibre}

\author{Y\MakeLowercase{utao} M\MakeLowercase{a} and X\MakeLowercase{ujia} M\MakeLowercase{eng}}
\address{School of Mathematical Sciences $\&$ Laboratory  of Mathematics and Complex Systems of Ministry of Education, Beijing Normal University, 100875 Beijing, China.} 
\email{mayt@bnu.edu.cn, 202321130122@mail.bnu.edu.cn} 

\begin{document}
\maketitle

\begin{abstract} Let $Z_1, \cdots, Z_n$ denote the eigenvalues of the product $\prod_{j=1}^{k_n} \boldsymbol{A}_j$, where $\{\boldsymbol{A}_j\}_{1 \le j \le k_n}$ are independent $n\times n$ complex Ginibre matrices. Define $\alpha = \lim\limits_{n \to \infty} \frac{n}{k_n}$. We prove that $X_n,$ a suitably rescaled version of $\max_{1 \le j \le n} |Z_j|^2,$ converges weakly as follows: to a non-trivial distribution $\Phi_\alpha$ for $\alpha \in (0, +\infty)$, to the Gumbel distribution when $\alpha = +\infty$, and to the standard normal distribution when $\alpha = 0$. This result reveals a phase transition at the boundaries of $\alpha$. Furthermore, we establish the exact rates of convergence in each regime.
\end{abstract} 

{\bf Keywords:} Ginibre ensemble; spectral radius;  Berry-Esseen bound; Wasserstein distance; phase transition.   

{\bf AMS Classification Subjects 2020:} 60G70, 60B20, 15B52

\section{Introduction}
The study of Products of random matrices, dating back to Bellman \cite{bellman} and \cite{FurKesten}, have garnered significant interest in recent research. Ipsen \cite{Ipsen} outlines several applications, including wireless telecommunication, disordered spin chains, the stability of large complex systems, quantum transport in disordered wires, symplectic maps, Hamiltonian mechanics, and quantum chromodynamics at non-zero chemical potential. Research has focused on products from various ensembles, such as Ginibre ensembles (\cite{AkBur, Bordenave, BurdaJW10, Forrester15, Ipsen, JQ17, JQ19, LW15, RS11}), spherical ensembles (\cite{ChangQ17, ChangLQ, Zeng16}), truncated unitary ensembles (\cite{JQ19, LW15}), and rectangular matrices (\cite{BJLN10, IpsenKie14, Tik, Zeng17}). For further information on products of independent random matrices and related topics, readers are referred to \cite{AkIp15, BaiY98, Bordenave, BouLa, Burda13, Crisanti12, GT, Kopel2020, LWW23, LW24, RS11, RRSV15}.

For a given integer $k \ge 1$, let $\boldsymbol{A}_1, \cdots, \boldsymbol{A}_{k}$ be independent and identically distributed (i.i.d.) $n \times n$ complex Ginibre matrices. Denote the eigenvalues of the product $\prod_{j=1}^k \boldsymbol{A}_j$ by $Z_1, \cdots, Z_n$. Akemann and Burda \cite{AkBur} derived the joint density function of $Z_1, \cdots, Z_n$. For fixed $k$, using methods related to determinantal point processes, Burda {\it et al}. \cite{BurdaJW10} and Burda \cite{Burda13} showed that the empirical distribution of $(Z_j / n^{k/2})_{1 \le j \le n}$ converges to a new distribution on the unit disk in the mean sense. G$\ddot{o}$tze and Tikhomirov \cite{GT} extended this result to the case where the matrix entries are i.i.d. random variables satisfying certain moment conditions. Bordenave \cite{Bordenave} and O'Rourke and Soshnikov \cite{RS11} strengthened this convergence to hold almost surely. Jiang and Qi \cite{JQ19} further obtained the limiting distribution for the empirical spectral distributions, in a framework that, for the first time, allows the number of factors $k_n$ to vary with the matrix size $n$.

Regarding the spectral radius, let  
$$\alpha_n = \frac{n}{k_n} \quad \text{and} \quad \alpha = \lim_{n \to \infty} \alpha_n.$$  
Jiang and Qi \cite{JQ17} established that, under appropriate scaling, the spectral radius of the product of Ginibre ensembles converges weakly to the Gumbel distribution, a distribution $\Psi_{\alpha}$, or the standard normal distribution, depending on whether $\alpha = +\infty$, $\alpha \in (0, +\infty)$, or $\alpha = 0$, respectively.  
However, $\Psi_{\alpha}$ itself does not exhibit phase transitions. This naturally leads to the question of how the scaling should be modified to reveal such transitions.

On the other hand, the first author, together with collaborators, has carried out a series of studies on the precise convergence rates of extreme eigenvalues for non-Hermitian random matrices toward the Gumbel distribution. These works analyze the convergence behavior under two metrics: the Berry-Esseen bound and the $W_1$-Wasserstein distance, as detailed in \cite{HuMa25, MaMeng2025, Ma2025c, MaWang25}. It is therefore also of interest to investigate the corresponding convergence rate for the spectral radius of the product of Ginibre ensembles.

To capture the phase transition behavior, we rescale the spectral radius by defining
\[
X_n = b_n^{-1}\left( \sqrt{\alpha_n}\left( \max_{1 \leq j \leq n} \log |Z_j|^2 - k_n \psi(n) \right) - a_n \right),
\]
where
\[
a_n = \sqrt{\log(\alpha_n + 1)} - \frac{\log\left(\sqrt{2\pi} \log\left(\alpha_n + e^{1/\sqrt{2\pi}}\right)\right)}{\sqrt{\log(\alpha_n + e)}}, \quad 
b_n = \frac{1}{\sqrt{\log(\alpha_n + e)}},
\]
and \(\psi(x) = \Gamma'(x)/\Gamma(x)\) is the digamma function, with \(\Gamma\) denoting the Gamma function.

We further define the limiting parameters
\[
a = \sqrt{\log(\alpha + 1)} - \frac{\log\left(\sqrt{2\pi} \log\left(\alpha + e^{1/\sqrt{2\pi}}\right)\right)}{\sqrt{\log(\alpha + e)}}, \quad 
b = \frac{1}{\sqrt{\log(\alpha + e)}},
\]
which correspond to the limits of \(a_n\) and \(b_n\), respectively.

For \(\alpha \in (0, \infty)\), we introduce
\[
v_\alpha(j, x) = \frac{j}{\sqrt{\alpha}} + a + b x,
\]
and define the distribution function via the infinite product:
\[
\Phi_\alpha(x) = \prod_{j=0}^{\infty} \Phi\left(v_\alpha(j, x)\right),
\]
where \(\Phi\) is the cumulative distribution function of the standard normal distribution.

Next, we define
\[
q_1(m, x) = \frac{1}{12\sqrt{\alpha}} \left( 2\alpha\left(v^2_\alpha(m, x) - 1\right) - 3\sqrt{\alpha}(2m + 1)v_\alpha(m, x) + 6m(m + 1) \right),
\]
and
\[
q_2(m, x) = c_1 - c_2 x - \frac{m}{2\alpha^{3/2}},
\]
where, letting \(w(t) = 2t \log t\),
\[
\begin{aligned}
c_1 &:= \frac{\sqrt{\log(\alpha + 1)}}{w(\alpha + 1)} + \frac{2}{ w(\alpha + e^{1/\sqrt{2\pi}})\sqrt{\log(\alpha + e)} } - \frac{\log(\sqrt{2\pi} \log(\alpha + e^{1/\sqrt{2\pi}}))}{w(\alpha + e) \sqrt{\log(\alpha + e)}}, \\
c_2 &:= \frac{1}{2(\alpha + e)\left(\log(\alpha + e)\right)^{3/2}}.
\end{aligned}
\]

We are now in a position to state our main results.

\begin{thm}\label{main}
    Let \( X_n \), \( \alpha \), \( \alpha_n \), \( v_\alpha \), \( q_{1}(m, \cdot) \), \( q_{2}(m, \cdot) \), and \( \Phi_\alpha \) be defined as above.  
    The Berry-Esseen bounds for the three regimes are given as follows:

    \textbf{(1) Case \( \alpha = 0 \).}  
    Let \( \beta := \lim_{n \to \infty} \frac{n^3}{k_n} \). The convergence rate depends on \( \beta \) in the following way:
    
    \begin{itemize}
        \item If \( \beta = +\infty \), then  
        \[
        \sup_{x \in \mathbb{R}} \left| \mathbb{P}(X_n \leq x) - \Phi(x) \right| = \frac{\sqrt{\alpha_n}}{\sqrt{2\pi}} (1 + o(1)).
        \]
        
        \item If \( \beta \in (0, +\infty) \), then  
        \[
        \sup_{x \in \mathbb{R}} \left| \mathbb{P}(X_n \leq x) - \Phi(x) \right| = \frac{(2\sqrt{\beta} + \sqrt{4\beta + 1})(1 + o(1))}{4\sqrt{2\pi e} \, n} \exp\left( \frac{2\sqrt{\beta}}{\sqrt{4\beta + 1} + 2\sqrt{\beta}} \right).
        \]
        
        \item If \( \beta = 0 \), then  
        \[
        \sup_{x \in \mathbb{R}} \left| \mathbb{P}(X_n \leq x) - \Phi(x) \right| = \frac{1 + o(1)}{4\sqrt{2\pi e} \, n}.
        \]
    \end{itemize}

    \textbf{(2) Case \( \alpha \in (0, \infty) \).}  
    Let \( \eta := \lim_{n \to \infty} (\alpha_n - \alpha) n \). The result depends on the asymptotic behavior of \( \eta \):
    
    \begin{itemize}
        \item If \( |\eta| = +\infty \), then  
        \[
        \sup_{x \in \mathbb{R}} \left| \mathbb{P}(X_n \leq x) - \Phi_\alpha(x) \right| = (1 + o(1)) |\alpha_n - \alpha| \sup_{x \in \mathbb{R}} \Phi_\alpha(x) \left| \sum_{m=0}^{\infty} \frac{q_2(m, x)\phi(v_\alpha(m, x))}{\Phi(v_\alpha(m, x))} \right|.
        \]
        
        \item If \( |\eta| \in (0, +\infty) \), then  
        \[
        \sup_{x \in \mathbb{R}} \left| \mathbb{P}(X_n \leq x) - \Phi_\alpha(x) \right| = \frac{1 + o(1)}{n} \sup_{x \in \mathbb{R}} \Phi_\alpha(x) \left| \sum_{m=0}^{\infty} \frac{\phi(v_\alpha(m, x))}{\Phi(v_\alpha(m, x))} \big( q_1(m, x) + \eta q_2(m, x) \big) \right|.
        \]
        
        \item If \( \eta = 0 \), then  
        \[
        \sup_{x \in \mathbb{R}} \left| \mathbb{P}(X_n \leq x) - \Phi_\alpha(x) \right| = \frac{1 + o(1)}{n} \sup_{x \in \mathbb{R}} \Phi_\alpha(x) \left| \sum_{m=0}^{\infty} \frac{q_1(m, x)\phi(v_\alpha(m, x))}{\Phi(v_\alpha(m, x))} \right|.
        \]
    \end{itemize}

    \textbf{(3) Case \( \alpha = +\infty \).}  
    \[
    \sup_{x \in \mathbb{R}} \left| \mathbb{P}(X_n \leq x) - e^{-e^{-x}} \right| = \frac{(\log \log \alpha_n)^2}{2e \log \alpha_n} (1 + o(1)).
    \]
\end{thm}

This theorem provides an answer to the second question concerning the convergence rate of the distribution of \( X_n \) to its limiting distributions. We next present the rate at which \( \Phi_\alpha \) converges to the Gumbel distribution as \( \alpha \to +\infty \), and to the standard normal distribution as \( \alpha \to 0^+ \). These results confirm the phase transition behavior of \( \Phi_\alpha \), as anticipated earlier.

\begin{rmk}\label{transition}
    As the tuning parameter \( \alpha \) varies from \( 0 \) to \( +\infty \), the following asymptotic bounds hold:
    \[
    \sup_{x \in \mathbb{R}} \left| \Phi_\alpha(x) - \Phi(x) \right| = \frac{\sqrt{\alpha}}{\sqrt{2\pi}}(1 + o(1)), \quad \text{for } \alpha \ll 1,
    \]
    and
    \[
    \sup_{x \in \mathbb{R}} \left| \Phi_\alpha(x) - e^{-e^{-x}} \right| = \frac{(\log \log \alpha)^2}{2e \log \alpha}(1 + o(1)), \quad \text{for } \alpha \gg 1.
    \]
\end{rmk}

\begin{rmk}
    We observe three distinct asymptotic regimes for both the cases \( \alpha = 0 \) and \( \alpha \in (0, +\infty) \). In fact, these asymptotic behaviors can be unified into a single expression in each case.

    \begin{enumerate}
        \item For \( \alpha = 0 \), we have
        \begin{equation}\label{alphasum}
            \sup_{x \in \mathbb{R}} \left| \mathbb{P}(X_n \leq x) - \Phi(x) \right| = (1 + o(1)) \sup_{x \in \mathbb{R}} \phi(x) \left| \sqrt{\alpha_n} - \frac{x}{4n} \right|.
        \end{equation}
        The right-hand side of \eqref{alphasum} exhibits three different asymptotic behaviors depending on whether \( \beta = 0 \), \( \beta \in (0, +\infty) \), or \( \beta = +\infty \).

        \item For \( \alpha \in (0, +\infty) \), it holds that
        \begin{equation}\label{alphafinite}
            \begin{aligned}
            & \sup_{x \in \mathbb{R}} \left| \mathbb{P}(X_n \leq x) - \Phi_\alpha(x) \right| \\
            & = (1 + o(1)) \sup_{x \in \mathbb{R}} \Phi_\alpha(x) \left| \sum_{m=0}^{\infty} \frac{\phi(v_\alpha(m, x))}{\Phi(v_\alpha(m, x))} \left( n^{-1} q_1(m, x) + (\alpha_n - \alpha) q_2(m, x) \right) \right|.
            \end{aligned}
        \end{equation}
    \end{enumerate}
\end{rmk}

\begin{rmk}\label{upperbound}
    The supremum expressions appearing in Theorem~\ref{main} for \( \alpha \in (0, +\infty) \) do not admit a closed form. We therefore provide simple-though not sharp-upper bounds as follows:
    \[
    \sup_{x \in \mathbb{R}} \Phi_\alpha(x) \left| \sum_{m=0}^{\infty} \frac{q_1(m, x) \phi(v_\alpha(m, x))}{\Phi(v_\alpha(m, x))} \right| \leq \frac{4}{3} \left( \alpha + \sqrt{\alpha} + 1 \right),
    \]
    and
    \[
    \sup_{x \in \mathbb{R}} \Phi_\alpha(x) \left| \sum_{m=0}^{\infty} \frac{q_2(m, x) \phi(v_\alpha(m, x))}{\Phi(v_\alpha(m, x))} \right| \leq \frac{2}{e \ln 2} \left( c_1 + \frac{c_2(\alpha - 1)}{b} + \frac{1}{\alpha} \right) (1 + \sqrt{\alpha}).
    \]
\end{rmk}

The key estimates employed in the proof of Theorem \ref{main} can be readily adapted to establish the convergence rate in the $W_1$-Wasserstein distance.

\begin{rmk}\label{w1cor}
		Under the same framework as Theorem \ref{main}, 
	the $ W_1 $-Wasserstein distances  are as follows.	
	\item[(1).] Case $\alpha = 0:$
	
		$$  W_{1}\left(\mathcal{L}(X_n),\Phi\right)=[\sqrt{\alpha_n}(2\Phi(4n\sqrt{\alpha_n})-1)+\frac{1}{2n}\phi(4n\sqrt{\alpha_n})](1+o(1)).$$
	
	\item[(2).] Case $\alpha \in (0,\infty):$
	$$  W_{1}\left(\mathcal{L}(X_n),\Phi_{\alpha}\right)=\int_{-\infty}^{+\infty}\Phi_\alpha(x) | \sum_{m=0}^{\infty} \frac{\phi(v_\alpha(m, x))}{\Phi(v_\alpha(m, x))} \left( n^{-1} q_1(m, x) + (\alpha_n - \alpha) q_2(m, x) \right) |dx.$$
	\item[(3).] Case $\alpha =+\infty:$	
	$$ W_{1}\left(\mathcal{L}(X_n),\Lambda\right)=\frac{(\log \log \alpha_n)^{2}}{2\log \alpha_n}(1+o(1)).$$
\end{rmk}

\begin{rmk}
It is worth noting that when $k_n = 1,$ the product of complex Ginibre ensembles reduces to a single complex Ginibre ensemble. In this case, the Berry-Esseen bound and the $W_1$-Wasserstein distance for the rescaled spectral radius relative to the Gumbel distribution are of order $O\left( \frac{(\log \log n)^2}{\log n} \right).$  This result differs from the order $O\left( \frac{\log \log n}{\log n} \right)$ obtained in \cite{MaMeng2025}. As elucidated in \cite{HuMa25, MaTian}, this discrepancy is attributed to the use of different rescaling coefficients for the spectral radius.
\end{rmk}

Given the complexity of the proof of Theorem \ref{main}, we give an outline to improve its accessibility. 

{\bf Sketch of the proof}. 
We begin by following the approach of \cite{MaMeng2025} and partition the real line into three intervals using thresholds $\ell_{1,\alpha}(n)$ and $\ell_{2,\alpha}(n)$, where $\alpha \in [0, +\infty]$. The asymptotic behavior of $\mathbb{P}(X_n \le x)$ is analyzed in detail for $x$ in the central interval $[-\ell_{1,\alpha}(n), \ell_{2,\alpha}(n)],$ which yields the precise convergence rate. Meanwhile, the contributions from the two tail intervals are shown to be negligible. We now describe the derivation of a precise asymptotic expression for $\mathbb{P}(X_n \le x)$ when $x \in [-\ell_{1,\alpha}(n), \ell_{2,\alpha}(n)]$.

For simplicity, define for $0 \le m \le n-1$:
\begin{equation}\label{cnmx}
c_n(m, x) = \mathbb{P}\left( \log Y_{n-m} > k_n \psi(n) + \alpha_n^{-1/2}(a_n + b_n x) \right),
\end{equation}
so that
\[
\mathbb{P}(X_n \le x) = \prod_{m=0}^{n-1} \left(1 - c_n(m, x)\right).
\]
 The proof strategy is tailored to the value of \(\alpha\):
\begin{enumerate}

\item Case \(\alpha = +\infty\):  Here, it includes the possibility that \(k_n = O(1)\). We simplify the term \(\log Y_{n-m}\) in the expression \eqref{cnmx} by approximating it as a sum of i.i.d. exponential variables, \(\sum_{i=1}^{(n-m)k_n} \xi_i\) and then applying the central limit theorem (Lemmas \ref{tailforsumnew} and \ref{333}) to give a precise asymptotic for \(c_n(m, x)\) for \(0 \le m \le j_n - 1\). The tail sum \(\sum_{m=j_n}^{n-1} c_n(m, x)\) is negligible, allowing us to express the distribution function of $X_n$ as $$\mathbb{P}(X_n \le x) = \exp(-\sum_{m=0}^{j_n-1} c_n(m, x))(1+o(1)),$$ where the sum provides the leading-order behavior.

\item Case \(\alpha \in (0, +\infty)\): An Edgeworth expansion (Lemma \ref{m1}) yields the asymptotics of \(1 - c_n(m, x)\) for \(0 \le m \le j_n - 1\) within the range \(-\ell_{1,\alpha}(n) \le x \le \ell_{2,\alpha}(n)\). This leads to the key approximations in Lemma \ref{jn1}:
    \[\aligned
    \mathbb{P}(X_n \le x)& = \prod_{m=0}^{j_n-1}(1 - c_{n}(m, x))(1+o(1)); \\
    \Phi_{\alpha}(x) &= \prod_{m=0}^{j_n-1} \Phi(v_{\alpha}(m, x))(1+o(1)), \endaligned 
    \]
    from which the Berry-Esseen bound is derived.

\item Case \(\alpha = 0\): The asymptotic analysis shows that (Lemma \ref{88})\[\prod_{m=1}^{n-1}(1 - c_{n}(m, x))=1 + o(1),\] so the dominant contribution comes from the precise asymptotics of \(1 - c_{n}(0, x)\) given in Lemma \ref{ed2}. The Berry-Esseen bound follows.

\end{enumerate}

The paper is structured as follows. Section 2 contains preliminary lemmas, including upper bounds for \(c_n(m, x)\) and relevant Edgeworth expansions. The detailed analysis of \(c_n(m, x)\) and the proof of Theorem \ref{main} are divided into Sections 3, 4, and 5, which address the cases \(\alpha = +\infty\), \(\alpha \in (0, +\infty)\), and \(\alpha = 0\), respectively.

 We adopt standard asymptotic notation: \(t_n = O(z_n)\) if \(t_n/z_n\) tends to a non-zero constant, and \(t_n = o(z_n)\) if the ratio tends to zero. The notation \(f \lesssim g\) for \(f, g \geq 0\) means \(f \leq C g\) for an absolute constant \(C>0\). We also write \(t_n \ll z_n\) (or \(z_n \gg t_n\)) to denote \(t_n/z_n \to 0\).

\section{Premininaries}
We commence this section by presenting key lemmas. The initial one is pivotal, as it permits a reduction to a much simpler, decoupled framework. This phenomenon was initially observed by Kostlan \cite{Kostlan1992} in the context of complex Ginibre ensembles and has subsequently been generalized to a wide range of complex non-Hermitian random matrices.  

The details are as follows:

\begin{lem}\label{ctor}
	Let \(Z_1, \cdots, Z_n\) be the eigenvalues of \(\prod_{j=1}^{k_n} \boldsymbol{A}_j\), and let \(\{S_{j, \,r},\ 1\le j\le n,\ 1\le r\le k_n\}\) be independent random variables such that \(S_{j, \,r}\) has density function \(y^{j-1}e^{-y}/(j-1)!\) for $y>0.$ Defining \(Y_j = \prod_{r=1}^{k_n} S_{j, \,r}\) for \(1 \le j \le n\), we have
 \[
 \max_{1\le j\le n} \log|Z_j|^2 \stackrel{d}{=} \max_{1\le j\le n} \log Y_j.
 \]

\end{lem}
 We first recall our target
\[X_n = b_n^{-1} \left( \sqrt{\alpha_n}( \max_{1 \leq j \leq n} \log |Z_j|^2 - k_n \psi(n)) - a_n \right).\]
It follows from Lemma \ref{ctor} that
$$
	\aligned \mathbb{P}\left(X_n\leq x\right)&=\mathbb{P}\big(\max_{1 \leq j \leq n} \log Y_j\leq k_n\psi (n)+\frac{a_n+b_nx}{\sqrt{\alpha_n}}\big)
	\endaligned 
	$$
and then by the property of maximum and the independence of $(Y_j)_{1\le j\le n}$, we have in further that \begin{equation}\label{begin}
	\aligned
	\mathbb{P}\left(X_n\leq x\right)=
	\prod_{j=1}^{n}\mathbb{P}\big( \log Y_j\leq k_n\psi (n)+\frac{a_n+b_nx}{\sqrt{\alpha_n}}\big)=\prod_{m=0}^{n-1}(1- c_n(m, x)).
	\endaligned 
\end{equation}
Here, $c_n(m, x)$ is given in \eqref{cnmx}.
To get asymptotical expression for $\mathbb{P}(X_n\le x),$ as mentioned in the introduction, we will apply central limit theorem for i.i.d. case or use the Edgeworth expansion. Both of them requires the moments of $\log S_{j, 1}$ for $1\le j\le n.$   

We first provide  properties on digamma function (\cite{Abramowitz1968}), which is involved in the expectations related to $\{\log S_{j,\,1}\}_{j=1}^{n}$ and then state the corresponding expectations. These two lemmas will serve as a fundamental tool for our subsequent analysis. 

\begin{lem}\label{diagammapro}  Let $\psi(x) = \Gamma'(x)/\Gamma(x)$ be the digamma function, with $\Gamma$ being the Gamma function .  
	\begin{enumerate}
	
		\item[\textup{(a)}.]  For any $s \geq 1$,
		\[
		\psi(j + s) - \psi(j) \leq \frac{s}{j}.
		\]
		
		\item[\textup{(b)}.]  For sufficiently large $z$, the following asymptotics hold:
		\begin{align*}
			\psi(z) &= \log z - \frac{1}{2z} + O(z^{-2}), \quad\quad
			\psi'(z) = \frac{1}{z} + \frac{1}{2z^2} + O(z^{-3});\\
			\psi^{(2)}(z)  &= -\frac{1}{z^2} - \frac{1}{z^3} + O(z^{-4}),\quad\quad
			\psi^{(3)}(z)  = \frac{2}{z^3} + \frac{3}{z^4} +  O(z^{-5}).
		\end{align*}
	\end{enumerate}
\end{lem}

Now, we directly state the expectations related to $\{\log S_{i,1}\}_{i=1}^{n}$ without proof. 
\begin{lem}\label{le}
	Let $(S_{j, 1})_{1\le j\le n}$ be defined as above. The following assertions hold:
	
	\item[(1).]  For the moments and moment generating function (MGF), we have:
	\begin{align*}
		\mu:=\mathbb{E}[\log S_{j, 1}] = \psi(j)&, \quad\quad
		\sigma^{2}:={\rm Var}[\log S_{j,1}] = \psi'(j), \\
		\mathbb{E}[S_{j, 1} \log S_{j, 1}] = j \psi(j+1)&, \quad\quad
		\mathbb{E}[e^{\lambda \log S_{j,1}}] = \frac{\Gamma(j+\lambda)}{\Gamma(j)}.
	\end{align*}
	
\item[(2).] For sufficiently large $j$, the skewness and kurtosis correction terms for $\log S_{j,1}$ are given by:
$$\gamma_{1}:=\frac{\mathbb{E}(\log S_{j,1}-\mu)^{3}}{6\sigma^{3}}=-\frac{1}{6\sqrt{j}}(1+O(j^{-1}))\quad \textup{(skewness correction)};$$
$$\gamma_{2}:=\frac{\mathbb{E}[(\log S_{j,1}-\mu)^4] - 3\sigma^4}{24\sigma^4} =\frac{1}{12j}(1+O(j^{-1}))\quad \textup{(kurtosis correction)}.$$
\end{lem}

Using the expectation of 
$ \log Y_j $ from Lemma \ref{le} and the Markov inequality, we derive both upper and lower tail probabilities for $ \log Y_j $.

\begin{lem}\label{H}
		Let  $Y_j$ be defined as above.  The following probability bounds hold:
	\item[(1).] For all $0<m\ll n$, $$	c_n(m,x)
	\leq\exp\left\{-\frac{m^2}{16\alpha_n}-\frac{m(a_n+b_n x)}{4\sqrt{\alpha_n}} \right\}.
	$$
	\item[(2).] When $m=0$ and any $x>0$, 
	$$c_n(0,x)\leq\exp\{-\frac{1}{4}(a_n+b_nx)^{2}\}.$$
	\item[(3).]For $m=0$ and any $ x $ satisfying $ a_n+b_nx<0,$
	$$	1-c_n(0,x)
	\leq\exp\left\{  - \frac{1}{3}(a_n+b_nx)^{2}\right\}.
	$$
\end{lem}
\begin{proof}
	 Set $j=n-m.$ Given any $t>0,$  it follows from Lemma \ref{le} and the Markov inequality that for each $0<m\ll n$ and any \( x > 0 \),
	$$\aligned 
	c_n(m,x)\leq&\mathbb{E}(e^{t \log Y_j})\exp \left\{ -t\left( k_n\psi (n)+\alpha_n^{-1/2}(a_n+b_nx) \right) \right\}\\
	=& \exp \{ k_n (\log\Gamma(j + t) - \log \Gamma(j)) - t \left(  k_n\psi (n)+\alpha_n^{-1/2}(a_n+b_nx)\right) \}\\
	=& \exp\{ k_n \int_0^t \psi(j + s) ds - t \left(  k_n\psi (n)+\alpha_n^{-1/2}(a_n+b_nx) \right)\}\\
	=&\exp\{ k_n \int_0^t (\psi(j + s) - \psi(j)) ds - tk_n (\psi (n) - \psi(j)) -t \alpha_n^{-1/2} (a_n+b_nx) \}.
	\endaligned $$
	Due to Lemma \ref{diagammapro}, it holds that
	$$\aligned \int_0^t \psi(j + s) - \psi(j) ds&=\int_0^1 \psi(j + s) - \psi(j) ds+\int_1^t \psi(j + s) - \psi(j) ds\\
	&\leq \psi(j + 1) - \psi(j)+\int_1^t \frac{s}{j} ds\\
	&\leq\frac{t^{2}+1}{2j}\leq\frac{t^2}{j} \endaligned $$ 
	and $$\psi (n) - \psi(j)=\log n-\log j-\frac{1}{2n}+\frac{1}{2j}+o(n^{-2})\ge \frac{m}{n}.$$
	Therefore, we have 
	\begin{equation}\label{k12}
		\aligned 
		c_{n}(m, x)	\leq\exp\left\{ k_nj^{-1}t^2-t k_n(\psi (n) - \psi(j))) -t\alpha_n^{-1/2}(a_n+b_nx) \right\}
		\endaligned
	\end{equation}	
	for all $ t > 1 $ and sufficiently large $ n. $ 
	Applying Lemma \ref{diagammapro} again, we know 
	 and then 
	$$t (t-j(\psi (n) - \psi(j)))=t(t-j(\log n-\log j-\frac{1}{2n}+\frac{1}{2j}+o(n^{-2})))\leq t(t-\frac{m}{2}).$$
	By selecting $ t=\frac{m}{4}$, we get
	$$	c_n(m,x)
	\leq\exp\left\{-\frac{m^2}{16\alpha_n}-\frac{m(a_n+b_n x)}{4\sqrt{\alpha_n}} \right\}.
	$$
	When $m=0$ and $x>0$,	it follows from the preceding proof that
	$$c_n(0,x)\leq \exp\{\frac{t^{2}}{\alpha_n}-\frac{t(a_n+b_n x)}{\sqrt{\alpha_n}}\}.$$
	Since $x>0,$ we take $t=(\frac{a_n+b_nx}{2})\sqrt{\alpha_n}$ and obtain
	$$c_n(0,x)\leq\exp\{-\frac{1}{4}(a_n+b_nx)^{2}\}.$$
	For any $x$ satisfying	$a_n+b_nx<0,$ the lower tail probability admits the same exponential bound, which implies, for any $ t<0,$ we have
	$$ \aligned
	1-c_n(0,x)&\leq \exp\left\{ k_n \int_0^t \psi(n + s) -\psi (n)ds - \frac{t(a_n+b_nx)}{\sqrt{\alpha_n}}\right\}\\
	&\leq\exp\left\{ \frac{2t^2}{3\alpha_n} - \frac{t(a_n+b_n x)}{\sqrt{\alpha_n}}\right\}.
	\endaligned
	$$
	Now $a_n+b_n x<0$, we set
	$ t=\sqrt{\alpha_n}(a_n+b_nx),$ which leads to
	$$ \aligned
	1-c_n(0,x)
	\leq\exp\left\{  - \frac{1}{3}(a_n+b_nx)^{2}\right\}.
	\endaligned
	$$ The proof is then completed. 
\end{proof}
By virtue of the properties of 
$\log S_{j, \,r}$
in Lemma \ref{le}, the Edgeworth expansion of 
$\log S_{j, \,r}$
can be derived (see \cite{DasGupta2008}; for the proof, see \cite{Esseen1945}). This expansion provides a more precise estimation compared to the central limit theorem.
\begin{lem}\label{ed}
	Given $0\leq m\ll n$ and $n\lesssim k_n.$ For any $|x_n|\lesssim n^{1/6},$
$$\mathbb{P}\left( \frac{\log Y_{n-m}-k_n\psi(n-m)}{\sqrt{k_n}\psi'(n-m)}\leq x_n\right)=\Phi(x_n)-\frac{(1-x_n^{2})}{6\sqrt{k_nn}}\phi(x_n)+O(k_n^{-\frac{3}{2}}+k_n^{-\frac{1}{2}}n^{-\frac{3}{2}}).$$
\end{lem} 
 \begin{proof}
 	Since \((\log S_{n-m, \,r})_{1\le r\le k_n}\) are i.i.d. random sequence and satisfy the Cram\'er condition and possesses a finite fourth moment, it fulfills the requirements for applying the Edgeworth expansion.

Hence,  
 	\begin{align*}&\mathbb{P}\left( \frac{\log Y_{n-m}-k_n\psi(n-m)}{\sqrt{k_n}\psi'(n-m)}\leq x_n\right) 
 	\\=&\mathbb{P}\left( \frac{\sum_{r=1}^{k_n}(\log S_{n-m,r}-\mathbb{E} \log S_{n-m,r})}{\sqrt{k_n {\rm Var}(\log S_{n-m, 1})}}\leq x_n\right)\\
 =&\Phi(x_n) +  \frac{\gamma_{1}(1-x_n^2)\phi(x_n)}{\sqrt{k_n}}  - \frac{\phi(x_n)}{k_n}( \gamma_2(x_n^3-3x_n)+\frac{\gamma_{1}^{2}}{72}(x_n^5-10x_n^3+15x_n))+O(k_n^{-\frac{3}{2}}),	\label{xn}\end{align*}
 where $\gamma_1$ and $\gamma_2$ are the skewness correction and the kurtosis correction of  $\log S_{n-m, 1},$ respectively. Lemma \ref{le} entails 
 $$\gamma_{1}=-\frac{1+O(n^{-1})}{6\sqrt{n}}\quad \text{and}\quad\gamma_{2}=\frac{1+O(n^{-1})}{12n}.$$
 Since $x^{k}\phi(x)$ is a bounded function for $0\leq k\leq 5$. For $|x_n|\leq n^{1/6},$ we see clearly that the third term is negligible with respect to the second term. 
 Therefore,
 $$\mathbb{P}\left( \frac{\log Y_{n-m}-k_n\psi(n-m)}{\sqrt{k_n}\psi'(n-m)}\leq x_n\right)=\Phi(x_n)-\frac{(1-x_n^{2})}{6\sqrt{k_nn}}\phi(x_n)+O(k_n^{-\frac{3}{2}}+k_n^{-\frac{1}{2}}n^{-\frac{3}{2}}).$$
 \end{proof}
 
At last, we borrow a summation property in \cite{MaMeng2025} as follows. 
\begin{lem}\label{sum} Let $\gamma_n$ be a positive sequence  
	and set $$\varsigma_n(m, x)=\frac{m}{\gamma_n}+a_n+b_n x$$ for $0\le m\le n-1.$  Given $L\geq 0$ and any $ x_{n}$ such that
	$ 1\ll \varsigma_{n}(L, x_{n})$ and let $c>0$ be a fixed constant. 
	\begin{enumerate} 
	\item When $\gamma_n$ satisfies $1\ll \gamma_n,$ we have 
	$$ \aligned \sum\limits_{m=L}^{+\infty}\varsigma_{n}^{-1}(m, x_{n})e^{- c\varsigma_{n}^{2}(m, x_{n})}&=\frac{\gamma_n e^{- c\varsigma_{n}^{2}(L, x_{n})}}{2c \varsigma_{n}^{2}(L, x_{n})}(1+ O(\varsigma_n^{-2}(L, x_n)+\varsigma_n(L, x_n)\gamma_n^{-1}))\\
	&\lesssim \frac{\gamma_n e^{- c\varsigma_{n}^{2}(L, x_{n})}}{ \varsigma_{n}^{2}(L, x_{n})} . \endaligned 
	$$	
	\item When $\gamma_n$ is bounded, 
	$$ \sum\limits_{m=L}^{+\infty}\varsigma_{n}^{-1}(m, x_{n})e^{- c\varsigma_{n}^{2}(L, x_{n})}\lesssim \frac{e^{- c\varsigma_{n}^{2}(L, x_{n})}}{\varsigma_{n}(L, x_{n})}.
	$$
	\end{enumerate}
\end{lem}
 
 \section{The case $ \alpha=+\infty.$} 
 
 In this section, we assume 
$\alpha=\lim_{n\to\infty}\frac{n}{k_n}=+\infty$. 
 Under this assumption, 
$k_n$ may be of order 
$ O(1) $. 
Recall
		$$c_{n}(m, x)=\mathbb{P}\left( \log Y_{n-m}> k_n\psi (n)+\frac{a_n+b_nx}{\sqrt{\alpha_n}}\right),$$ where 
		$\log Y_{n-m}=\sum_{r=1}^{k_n} \log S_{n-m, r}.$   
		Even $\log Y_{n-m}$ is the sum of i.i.d. random variables, we can not use the central limit theorem directly because $k_n$ could be a finite constant. We plan to reduce $\log S_{j, r}$ to $S_{j, r}$ since
		\begin{equation}\label{newsum}\sum_{r=1}^{k_n} S_{j, \,r}\stackrel{d}{=}\sum_{i=1}^{j k_n}\xi_i,\end{equation} 
		where $(\xi_i)_{i\ge 1}$  are i.i.d. exponential distribution with parameter $1.$ 

The expressions \eqref{A} and \eqref{B} below help us to understand that the terms $\log Y_{n-m}$ in $c_n(m, x)$ could be reduced to $\sum_{r=1}^{k_n} (\frac{S_{n-m, r}}{n-m}-1)$ and next we work on the tail probability of  $\sum_{r=1}^{k_n} (\frac{S_{n-m, r}}{n-m}-1).$ 
\begin{lem}\label{tailforsumnew}
	Let $\{S_{j,r}:  1\leq j\leq n, 1\leq r\leq k_n\}$ and $\alpha_n$ be defined as above. Set
	$$u_n(m, x) = \frac{m}{\sqrt{\alpha_n}} + a_n + b_n x.$$
For any $x$ and $m$ such that $1 \ll u_n(m, x) \ll n^{1/6},$ we have
	$$\aligned 
	&\mathbb{P} \left( \sum_{r=1}^{k_n}( \frac{S_{n-m, \,r}}{n-m}-1) >k_n(\psi (n)-\log (n-m))+\frac{a_n+b_n x}{\sqrt{\alpha_n}} \pm\alpha_n^{-\frac{3}{5}}\right)\\
	\leq&\frac{1}{u_n(m,x)}e^{-\frac{3u^{2}_n(m, x)}{8}}. 
	\endaligned $$
	Furthermore, for $x$ and $m$ such that $1\ll u_n(m, x)\lesssim \sqrt{\log \alpha_n},$ we have 
	\begin{equation}\label{z}
		\aligned 
		&\mathbb{P} \left( \sum_{r=1}^{k_n} \frac{S_{n-m, r}-(n-m)}{n-m} >k_n(\psi (n)-\log (n-m))+\frac{a_n+b_n x}{\sqrt{\alpha_n}}\pm \alpha_n^{-\frac{3}{5}}\right)\\
		=&\frac{1+O(u_n^{-2}(m, x))}{\sqrt{2\pi}u_n(m, x)}e^{-\frac{u^{2}_n(m, x)}{2}}. 
		\endaligned 
	\end{equation}
\end{lem}
\begin{proof}
	Let $\{\xi_i\}$ be i.i.d. random variables obeying an exponential distribution with parameter  $1$ and then $$\sum_{r=1}^{k_n} S_{n-m, \,r}\stackrel{d}{=}\sum_{i=1}^{(n-m)k_n}\xi_i.$$   
	It follows that
	$$\aligned 
	&\mathbb{P} \left( \sum_{r=1}^{k_n} (\frac{S_{n-m,r}}{n-m}-1) >k_n(\psi (n)-\log (n-m))+\frac{a_n+b_n x}{\sqrt{\alpha_n}}\pm \alpha_n^{-\frac{3}{5}}\right)\\
	=&\mathbb{P} \left(\frac{\sum_{i=1}^{(n-m)k_n} (\xi_i-1)}{\sqrt{(n-m)k_n}}  >\widetilde{u}_n(m,x)\right),
	\endaligned $$
	where $$ \widetilde{u}_n(m,x)=\sqrt{(n-m)k_n} (\psi (n)-\log (n-m)) + \sqrt{\frac{n-m}{n}}(a_n+b_n x) \pm\sqrt{\frac{n-m}{n}} \alpha_n^{-\frac1{10}}.$$ 
	Since $m\ll n^{2/3}$ and $k_n\ll n,$ we know $\sqrt{n-m}=\sqrt{n}(1+O(m n^{-1}))$ and then  we have 
	\begin{equation}\label{tildeuu}\aligned 
	\widetilde{u}_n(m, x)&=\left(\frac{m-\frac{1}{2}}{\sqrt{\alpha_n}} + a_n+b_nx\pm \alpha_n^{-\frac1{10}}\right)\left(1+O(\frac{m}{n})\right)\\
	&=u_n(m, x)\left(1+O(\alpha_n^{-\frac{1}{10}}u_n^{-1}(m,x)+\frac{m}{n})\right)\\
	&=u_n(m, x)(1+o(1)).
	\endaligned \end{equation} 
		For $1\ll u_n(m, x)\ll n^{\frac{1}{6}},$ so is $\widetilde{u}_n(m, x),$ 
	Theorem 1  from \cite{Petrov1975}  entails that
	\begin{equation}\label{prec}\mathbb{P} \left(\frac{\sum_{i=1}^{(n-m)k_n} (\xi_i-1)}{\sqrt{(n-m)k_n}}  >\widetilde{u}_n(m, x)\right)=(1-\Phi(\widetilde{u}_n(m, x)))(1+O((nk_n)^{-1/2}u_n^{3}(m,x))).\end{equation}
Recall the Mills ratio
$$	 1 - \Phi(t) = \frac{1}{\sqrt{2\pi}\,t} e^{-t^2/2} \left(1 + O(t^{-2})\right).$$
 Eventually, uniformly on $1 \ll u_n(m, x) \ll n^{1/6},$ we have from \eqref{tildeuu} and \eqref{prec}  that 
	$$\aligned \mathbb{P} \left(\frac{\sum_{i=1}^{(n-m)k_n} (\xi_i-1)}{\sqrt{(n-m)k_n}}  >\tilde{u}_n(m,x)\right)&=\frac{1+O(u_n^{-2}(m, x)+(nk_n)^{-1/2}u_n^{3}(m,x))}{\sqrt{2\pi}\widetilde{u}_n(m,x)}e^{-\frac{\widetilde{u}^{2}_n(m, x)}{2}}\\
	&\leq \frac{1}{u_n(m,x)}e^{-\frac{3u^{2}_n(m, x)}{8}}.\endaligned$$
	Particularly, if $1\ll u_n(m, x)\lesssim \sqrt{\log \alpha_n},$ which implies $m\lesssim \sqrt{\alpha_n\log\alpha_n},$ then $$u_n^{2}(m,x)(\alpha_n^{-\frac{1}{10}}u_n^{-1}(m,x)+\frac{m}{n})\lesssim u_n(m, x) \alpha_n^{-1/10}=o(1).$$ Thereby, 
	$$\exp(-\frac{1}{2}\widetilde{u}_n^2(m, x))=\exp(-\frac12 u_n^2(m, x))(1+O(u_n(m, x)\alpha_n^{-1/10}))$$ and 
	$$\frac{1+O(u_n^{-2}(m, x)+(nk_n)^{-1/2}u_n^{3}(m,x))}{\widetilde{u}_n(m,x)}=\frac{1+O(u_n^{-2}(m, x))}{u_n(m, x)}.$$ 
	Thus, we have 
	$$\mathbb{P}\left(\frac{\sum_{i=1}^{(n-m)k_n} (\xi_i-1)}{\sqrt{(n-m)k_n}}  >\tilde{u}_n(m,x)\right)=\frac{1+O(u_n^{-2}(m, x))}{\sqrt{2\pi} u_n(m, x)} e^{-\frac{1}{2}u_n^2(m, x)}.$$
	The proof is then completed. 
\end{proof}
Using this lemma, we give a precise asymptotic for $c_{n}(m, x).$ \begin{lem}\label{333}
Set $j_n=[\frac{1}{5}\sqrt{\alpha_n\log \alpha_n}].$  We get
$$c_n(m, x)=\frac{1+O((\log \alpha_n)^{-1})}{\sqrt{2\pi}u_n(m,x)}e^{-\frac{u^{2}_n(m, x)}{2}},$$
uniformly on $ 0\leq m\leq j_n$ and $  |x|\lesssim\log\log \alpha_n.$
\end{lem}
\begin{proof}
		
		Indeed, with $j:=n-m,$ $c_n(m, x)$ can be rewritten as 
		$$\aligned c_{n}(m, x)&=\mathbb{P}\big(\sum_{r=1}^{k_n}\frac{S_{j, r}-j}{j}>k_n(\psi (n)-\log j)+\frac{a_n+b_nx}{\sqrt{\alpha_n}}-\sum_{r=1}^{k_n}(\log \frac{S_{j,r}}{ j}-\frac{S_{j,r}-j}{j})\big).\endaligned$$
		Setting
	$$B_{\varepsilon}:=\{|\sum_{r=1}^{k_n}\log \frac{S_{j,r}}{ j}-\frac{S_{j,r}-j}{j}|\ge \varepsilon\}$$ for some $\varepsilon>0$ to be determined later, 
	we have 
	\begin{equation}\label{A}
		\aligned c_{n}(m, x)&\le \mathbb{P}(\sum_{r=1}^{k_n}\frac{S_{j,r}-j}{j}>k_n(\psi (n)-\log j)+\frac{a_n+b_nx}{\sqrt{\alpha_n}}-\varepsilon)+\mathbb{P}(B_{\varepsilon})\endaligned
	\end{equation}
	and 
	\begin{equation}\label{B}
		\aligned c_{n}(m, x)&\ge \mathbb{P}(\sum_{r=1}^{k_n}\frac{S_{j,r}-j}{j}>k_n(\psi (n)-\log j)+\frac{a_n+b_nx}{\sqrt{\alpha_n}}+\varepsilon).\endaligned
	\end{equation}
Let $\varepsilon=\alpha_n^{-3/5}.$ The conditions on $j_n$ and $x$ ensure that 
$$1\ll u_n(m, x)\le u_n(j_n, \log\log \alpha_n)=\frac{6}{5}\sqrt{\log \alpha_n}(1+o(1)),$$ which suits the requirement of Lemma \ref{tailforsumnew}. Therefore, 
\begin{equation}\label{equaforsumn} \aligned 
	\mathbb{P}(\sum_{r=1}^{k_n}\frac{S_{j,r}-j}{j}>k_n(\psi (n)-\log j)+\frac{a_n+b_nx}{\sqrt{\alpha_n}}\pm\varepsilon)=\frac{1+O(u_n^{-2}(m, x))}{\sqrt{2\pi}u_n(m, x)}\exp(-\frac12 u_n^2(m, x)).
	\endaligned 
\end{equation}
The right hand side of \eqref{equaforsumn} is decreasing on $u_n(m, x),$ and then  
\begin{equation}\label{tailp}\mathbb{P}(\sum_{r=1}^{k_n}\frac{S_{j,r}-j}{j}>k_n(\psi (n)-\log j)+\frac{a_n+b_nx}{\sqrt{\alpha_n}}\pm\varepsilon)\ge (\log\alpha_n)^{-1/2} \alpha_n^{-18/25}. \end{equation}

	We now give an upper bound for $B_{\varepsilon}.$

	By employing the triangle inequality property of absolute values, we first obtain 
	$$\aligned &\mathbb{P}\left(B_{\varepsilon}\right)
	\leq&	\mathbb{P}\left(\left|\sum_{r=1}^{k_n} (\log \frac{S_{j, r}}{j} - \frac{S_{j, r}-j}{j}-k_n (\psi (j)-\log j))\right| \geq \varepsilon-k_n|\psi (j)-\log j|\right). \endaligned $$
	An application of Markov's inequality, together with fundamental properties of variance and the condition $\varepsilon > k_n|\psi(j)-\log j|$, leads to the conclusion that
	$$ \mathbb{P}\left(B_{\varepsilon}\right)	
	\leq\frac{k_n{\rm Var}(\log  S_{j,1} - \frac{S_{j,1}}{j})}{(\epsilon-k_n|\psi(j)-\log j|)^{2}}.  $$	
		Lemma \ref{le} tells  
	$$\aligned{\rm Var}(\log  S_{j, 1} - \frac{S_{j, \,1}}{j})&={\rm Var}(\log  S_{j, 1})+j^{-2}{\rm Var}(S_{j, 1})-2j^{-1} \mathbb{E}((S_{j, 1}-j)\log S_{j, 1})\\
	&=\psi'(j)+j^{-1}-2(\psi(j+1)-\psi(j)) \endaligned $$ and then we have by Lemma \ref{diagammapro} that 
	$$\aligned {\rm Var}(\log  S_{j, 1} - \frac{S_{j, \,1}}{j})=\frac{1}{2j^2}+O(j^{-3}).\endaligned $$ 
	Since $\varepsilon=\alpha_n^{-3/5}\gg k_n|\psi(j)-\log j|,$ we  derive 
	\begin{equation}\label{upforb}\mathbb{P}(B_{\varepsilon})\lesssim n^{-4/5} k_n^{-1/5}.\end{equation}
Comparing the right hand sides of \eqref{tailp} and  \eqref{upforb}, we know 
$\mathbb{P}(B_{\varepsilon})$ is negligible in both \eqref{A} and \eqref{B}, and then we obtain
	$$c_n(m, x)=\frac{1+O((\log \alpha_n)^{-1})}{\sqrt{2\pi}u_n(m,x)}e^{-\frac{u^{2}_n(m, x)}{2}}$$
	uniformly on $0\le m\le j_n$ and $|x|\lesssim\log\log \alpha_n.$ The proof is completed now. 
	\end{proof}
Now, we will give the proof of Theorem \ref{main} for $\alpha=+\infty.$
 
  \begin{proof}[{\bf Proof of Theorem \ref{main} for $\alpha=+\infty$}] 
  
  Recall 
 \[a_n = \sqrt{\log\left(\alpha_n + 1\right)} - \frac{\log\left(\sqrt{2\pi}\log\left(\alpha_n + e^{\frac{1}{\sqrt{2\pi}}}\right)\right)}{\sqrt{\log\left(\alpha_n + e\right)}} \quad \text{and}\quad b_n=\frac{1}{\sqrt{\log\left(\alpha_n + e\right)}}.\]
 	 Following the partiton in our previous work \cite{MaMeng2025}, we take $ \ell_{1,\infty}(n)=\frac{1}{2}\log\log \alpha_n$ and $\ell_{2,\infty}(n)=\log\left(\sqrt{2\pi}\log (\alpha_n+e^{\frac{1}{\sqrt{2\pi}}})\right).$ 
  Now, we state first a precise estimate for $$|\mathbb{P}(X_n\leq x)-e^{-e^{-x}}|$$ uniformly on $x\in [-\ell_{1, +\infty}(n), \ell_{2, +\infty}(n)],$ which is 
 \begin{equation}\label{keyprecise}
 	|\mathbb{P}(X_n\leq x)-e^{-e^{-x}}|=e^{-e^{-x}}e^{-x}\frac{(1+o(1))}{2\log \alpha_n}\left|1-\frac{2x+4}{\ell_{2,\infty}(n)}+\frac{x^2+4x}{\ell_{2,\infty}^2(n)}\right|\end{equation} 
 	for $n$ sufficiently large.  

 	First, we explore the exact asymptotical expression of $\mathbb{P}(X_n\leq x),$ which is given by   
 	$$\mathbb{P}(X_n\leq x)=\exp(\sum_{m=0}^{n-1}\log(1-c_n(m, x)). $$
 	The monotonicity of $c_n(m, x)$ in both $m$ and $x$ implies that
 	$$c_n(m,x)\leq c_n(0,-\ell_{1,\infty}(n)),$$
 	uniformly
 	on  $ x\geq-\ell_{1,\infty}(n)$ and $ m\geq0.$ By definition, 
 	$$u_n(0, -\ell_{1, \infty}(n))=\sqrt{\log(\alpha_n+1)}-\frac{\ell_{1, \infty}(n)+\ell_{2, \infty}(n)}{\sqrt{\log(\alpha_n+e)}},$$
 	 and then Lemma \ref{333} entails 
 	$$c_n(0, -\ell_{1, \infty}(n))=O(\alpha_n^{-1/2}\log\alpha_n).$$
 	Applying the Taylor expansion $\log(1 - t) = -t (1 + O(t))$ for sufficiently small $|t|,$ we get
 	 $$\beta_n(x):=\sum_{m=0}^{n-1}\log(1-c_n(m,x))=-\sum_{m=0}^{n-1}c_n(m,x)(1+O(\alpha_n^{-1/2}\log\alpha_n)).$$ 
Choose $j_n=[\frac{1}{5}\sqrt{\alpha_n\log\alpha_n}]$ and $t_n=[8\sqrt{\alpha_n\log n}].$ From the monotonicity of $c_n(m,x)$ with respect to $ m $, we obtain
 	 \begin{equation}\label{32}
 	 	\sum_{m=0}^{j_n-1}c_n(m,x)\leq\sum_{m=0}^{n-1}c_n(m,x)\leq\sum_{m=0}^{t_n-1}c_n(m,x)+nc_n(t_n,x).
 	 \end{equation} 
 	 Lemma \ref{H} entails that 
 	 \begin{equation}\label{33}
 	 	nc_n(t_n,x)\leq  n\exp\{-4\log n-2\sqrt{\log n}(a_n+b_nx)\}=o(n^{-3}),
 	 \end{equation}
 	 uniformly
 	 on  $-\ell_{1,\infty}(n)\leq x\leq\ell_{2,\infty}(n).$ From Lemmas  \ref{sum} and \ref{333}, it follows that
 	 $$	\aligned\sum_{m=j_n}^{t_n-1}c_n(m,x)&\leq\sum_{m=j_n}^{t_n-1}\frac{1}{u_n(m,x)}e^{-\frac{3u^{2}_n(m, x)}{8}}\lesssim\frac{\sqrt{\alpha_n}}{ u_{n}^{2}(j_n, -\ell_{1,\infty}(n))}\exp(- \frac{3u_{n}^{2}(j_n, -\ell_{1,\infty}(n))}{8}).	\endaligned$$
 	Note that
 	 $$u_{n}(j_n, -\ell_{1,\infty}(n))=\frac{6}{5}\sqrt{\log \alpha_n}+o(1)$$  and $$u_{n}^{2}(j_n, -\ell_{1,\infty}(n))=\frac{36}{25}\log \alpha_n-\frac{12}{5}(\ell_{1, \infty}(n)+\ell_{2, \infty}(n))+O(1).$$
Thereby,
 	 \begin{equation}\label{34}
 	 	\sum_{m=j_n}^{t_n-1}c_n(m,x)\lesssim \alpha_n^{-\frac{1}{25}}(\log \alpha_n)^{\frac{7}{20}}.
 	 \end{equation}
 	 For $ 0\leq m\leq j_n-1,$ Lemmas \ref{sum} and \ref{333} immediately imply that
 	 \begin{equation}\label{35}
 	 	\sum_{m=0}^{j_n-1}c_n(m,x) =\sum_{m=0}^{j_n-1}\frac{1+O((\log \alpha_n)^{-1})}{\sqrt{2\pi}u_n(m,x)}e^{-\frac{u_n^{2}(m,x)}{2}}\\
 	 	=\frac{\sqrt{\alpha_n}(1+O((\log \alpha_n)^{-1}))}{\sqrt{2\pi } u_{n}^{2}(0, x)}e^{- \frac{u_{n}^{2}(0, x)}{2}},
 	 \end{equation}
 	 where for the last equality we use the fact$$\frac{1}{u_n^2(j_n, x)} e^{-\frac{u_n^2(j_n, x)}{2}} \ll \frac{1}{\log \alpha_n u_n^2(0, x)} e^{-\frac{u_n^2(0, x)}{2}}.$$
 	It is ready to see 
 	\begin{equation}\label{e4}
 	\aligned u_n(0,x)&=\sqrt{\log (\alpha_n+1)}+ \frac{x-\ell_{2,\infty}(n)}{\sqrt{\log (\alpha_n+e)}}\\
 	&=\sqrt{\log (\alpha_n+e)}+\frac{x-\ell_{2,\infty}(n)}{\sqrt{\log (\alpha_n+e)}}+O((\log \alpha_n)^{-1/2}(\alpha_n)^{-1}) , \endaligned 
 	\end{equation}
 	 and then 
 	 \begin{equation}\label{e3} \aligned 
 	 	u_n^2(0, x)&=\log (\alpha_n+e)-2\ell_{2,\infty}(n)+2 x+\frac{(x-\ell_{2,\infty}(n))^{2}}{\log (\alpha_n+e)}+O(\alpha_n^{-1});
 	 	\\
 	 	\exp(-\frac{1}{2}u_n^2(0, x))&=\sqrt{\frac{2\pi }{\alpha_n+e}}\log (\alpha_n+e^{\frac1{\sqrt{2\pi}}})\exp(-x-\frac{(\ell_{2,\infty}(n)-x)^{2}}{2\log(\alpha_n+e)})e^{O(\alpha_n^{-1})}\\
 	 	&=\sqrt{\frac{2\pi }{\alpha_n}}\log (\alpha_n)\exp(-x-\frac{(\ell_{2,\infty}(n)-x)^{2}}{2\log(\alpha_n+e)})(1+O(\alpha_n^{-1})).
 	 	\endaligned
 	 \end{equation}
 	 Putting \eqref{33}, \eqref{34}  \eqref{35} and \eqref{e3} back into \eqref{32}, we see
 	 $$\sum_{m=0}^{n-1}c_{n}(m,x)=\frac{1+O((\log  \alpha_n)^{-1})}{(1+\frac{x-\ell_{2,\infty}(n)}{\log (\alpha_n+e)})^{2}}\exp(-x-\frac{(x-\ell_{2,\infty}(n))^2}{2\log  (\alpha_n+e)}) $$ and thereby 
 	 \begin{equation}\label{betanx} \beta_n(x)=-\frac{1+O((\log  \alpha_n)^{-1})}{(1+\frac{x-\ell_{2,\infty}(n)}{\log (\alpha_n+e)})^{2}}\exp(-x-\frac{(x-\ell_{2,\infty}(n))^2}{2\log  (\alpha_n+e)}).	
 	 \end{equation}
The expression \eqref{betanx} guarantees that 
$$\aligned |\beta_n(x)+e^{-x}|&=e^{-x}|1-\frac{1+O((\log \alpha_n)^{-1})}{(1+\frac{x-\ell_{2,\infty}(n)}{\log (\alpha_n+e)})^{2}} \exp(-\frac{(x-\ell_{2,\infty}(n))^2}{2\log  (\alpha_n+e)})|\\
&=\frac{e^{-x}(1+O((\log \alpha_n)^{-1}))}{(1+\frac{x-\ell_{2,\infty}(n)}{\log (\alpha_n+e)})^{2}}|\frac{(x-\ell_{2,\infty}(n))^2}{2\log(\alpha_n+e)}+\frac{2(x-\ell_{2,+\infty}(n))}{\log(\alpha_n+e)}|. \endaligned $$
The choices of $\ell_{i, \infty}(n)$ are designed to ensure $|\beta_n(x)+e^{-x}|=o(1),$ whence 
 	 \begin{equation}\label{aphagg1}\aligned |\mathbb{P}(X_n\leq x)-e^{-e^{-x}}|&=e^{-e^{-x}}|\exp(\beta_n(x)+e^{-x})-1| \\
 	 &=e^{-e^{-x}}(1+o(1))|\beta_n(x)+e^{-x}|\\
 	 &=e^{-x-e^{-x}} \frac{|(x-\ell_{2,\infty}(n))^2+4(x-\ell_{2,\infty}(n))|}{2\log(\alpha_n+e)}(1+o(1))
 	 \endaligned \end{equation}
 	 uniformly on $[-\ell_{1, \infty}(n), \ell_{2, \infty}(n)].$
We see clearly  $$\sup_{x\in\mathbb{R}} e^{-x-e^{-x}} |x^k|<+\infty, \quad k=1, 2$$ and $\sup_{x\in\mathbb{R}} e^{-x-e^{-x}} =e^{-1},$ which together with \eqref{aphagg1} imply that \begin{equation}\label{midd}
\sup_{x\in [-\ell_{1, \infty}(n), \ell_{2, \infty}(n)]}|\mathbb{P}(X_n\leq x)-e^{-e^{-x}}|=\frac{(\log\log \alpha_n)^2}{2e\log\alpha_n}(1+o(1)).
\end{equation}
While for the two side intervals, we have 
	$$\sup\limits_{x\in(-\infty, -\ell_{1,\infty}(n)]}|\mathbb{P}(X_n\leq x)-e^{-e^{-x}}|\leq\mathbb{P}(X_n\leq -\ell_{1,\infty}(n))+e^{-e^{\ell_{1,\infty}(n)}}\lesssim e^{-e^{\ell_{1,\infty}(n)}}$$
	and
	$$\sup\limits_{x\in[\ell_{2,\infty}(n),+\infty)}|\mathbb{P}(X_n\leq x)-e^{-e^{-x}}|\leq1-\mathbb{P}(X_n\leq \ell_{2,\infty}(n))+1-e^{-e^{-\ell_{2,\infty}(n)}}\lesssim1-e^{-e^{-\ell_{2,\infty}(n)}}.$$
	Thereby, 
	\begin{equation}\label{twoside}\sup\limits_{x\in(-\infty, -\ell_{1,\infty}(n)] \cup [\ell_{2,\infty}(n),+\infty) }|\mathbb{P}(X_n\leq x)-e^{-e^{-x}}|\lesssim \frac{1}{\log \alpha_n}.
	\end{equation}
Combining \eqref{midd} and \eqref{twoside}, we derive 
	 	$$\sup_{x\in \mathbb{R}}|\mathbb{P}(X_n\le x)-e^{-e^{-x}}|= \frac{(\log \log \alpha_n)^{2}}{2e\log \alpha_n}(1+o(1)).$$
	 	The proof of Theorem \ref{main} for $\alpha=+\infty$
is completed. 	 	\end{proof}

 \section{The case $\alpha\in(0,\infty).$}
This section proceeds under the assumption that $ \alpha = \lim_{n\to\infty} \frac{n}{k_n}$ exists and lies in $(0, \infty)$, implying that $k_n = O(n)$. We first present an overview of the main stages of the proof and the complete technical details will be provided subsequently.
 
 Recall that 
 $$c_n(m, x)=\mathbb{P}(\log Y_{n-m}>k_n\psi (n)+\frac{a_n+b_n x}{\sqrt{\alpha_n}})$$ and $$u_n(m, x)=\frac{m}{\sqrt{\alpha_n}}+a_n +b_n x,\quad v_\alpha(m,x)=\frac{m}{\sqrt{\alpha}}+a+b x \quad \text{and} \quad \Phi_{\alpha}(x)=\prod_{m=0}^{\infty}\Phi(v_{\alpha}(m, x)).$$

Given that \( k_n \gg 1 \) and \( u_n(m, x) \) remains finite for finite \( m \) and \( x \), we apply an Edgeworth expansion to \( \log Y_{n-m} \) to derive a precise asymptotic expression for \( c_n(m, x) \). This expression holds uniformly over \( x \in [-\ell_{1,\alpha}(n), \ell_{2,\alpha}(n)] \) and \( 0 \le m \le j_n - 1 \), where the parameters \( \ell_{1,\alpha}(n) \), \( \ell_{2,\alpha}(n) \), and \( j_n \) will be specified later. Consequently,
\begin{equation}\label{totalmiddle}
\begin{aligned}
\left|\mathbb{P}(X_n \le x) - \Phi_\alpha(x)\right| = \Phi_\alpha(x) | \exp( \beta_n(x) - \sum_{m=0}^{+\infty} \log \Phi(v_\alpha(m, x)) ) - 1 |,
\end{aligned}
\end{equation}
where \( \beta_n(x) = \sum_{m=0}^{n-1} \log(1 - c_n(m, x)) \).

For the exponent term \( \beta_n(x) - \sum_{m=0}^{+\infty} \log \Phi(v_\alpha(m, x)) \), we aim to show that
\begin{equation}\label{o1}
\beta_n(x) - \sum_{m=0}^{+\infty} \log \Phi(v_\alpha(m, x)) = o(1),
\end{equation}
and more precisely,
\begin{equation}\label{o2}
\begin{aligned}
\beta_n(x) - \sum_{m=0}^{+\infty} \log \Phi(v_\alpha(m, x))
&=  \sum_{m=0}^{j_n - 1} \log\left( 1 + \frac{1 - c_n(m, x) - \Phi(v_\alpha(m, x))}{\Phi(v_\alpha(m, x))} \right) +o(1)\\
&=  \sum_{m=0}^{j_n - 1} \frac{1 - c_n(m, x) - \Phi(v_\alpha(m, x))}{\Phi(v_\alpha(m, x))}+o(1).
\end{aligned}
\end{equation}

Furthermore, we will establish that
\begin{equation}\label{o3}
1 - c_n(m, x) - \Phi(v_\alpha(m, x)) = d_\alpha(m, x)(1 + o(1)) \quad \text{and} \quad \sum_{m=0}^{+\infty} \frac{|d_\alpha(m, x)|}{\Phi(v_\alpha(m, x))} < +\infty.
\end{equation}

It follows that
\[
\left|\mathbb{P}(X_n \le x) - \Phi_\alpha(x)\right| = (1 + o(1)) \, \Phi_\alpha(x) \sum_{m=0}^{+\infty} \frac{|d_\alpha(m, x)|}{\Phi(v_\alpha(m, x))}
\]
uniformly for \( x \in [-\ell_{1,\alpha}(n), \ell_{2,\alpha}(n)] \).

As evident from the above, equations \eqref{o2} and \eqref{o3} are crucial to the proof, both of which require precise asymptotics for \( c_n(m, x) \). We therefore begin the proof by deriving a detailed estimate for \( c_n(m, x) \).

We start by writing \(\log Y_{n-m} = \sum_{r=1}^{k_n} \log S_{n-m, \,r}\). This expression represents a sum of i.i.d. random variables \(\{\log S_{n-m, r}\}_{1 \le r \le k_n}\). Substituting \(j = n - m\), we reformulate \(c_n(m, x)\) as follows:
\[
c_n(m, x) = \mathbb{P}\left( \frac{\log Y_j - k_n \psi(j)}{\sqrt{k_n \psi'(j)}} > \frac{k_n (\psi(n) - \psi(j))}{\sqrt{k_n \psi'(j)}} + \frac{a_n + b_n x}{\sqrt{n \psi'(j)}} \right).
 \]
 Now, define the function
 \begin{equation}\label{vndef}
v_n(m, x) = \frac{k_n (\psi(n) - \psi(n-m))}{\sqrt{k_n \psi'(n-m)}} + \frac{a_n + b_n x}{\sqrt{n \psi'(n-m)}}.
\end{equation}
An application of Lemma \ref{diagammapro} yields the approximation $$v_n(m, x) = u_n(m, x)(1 + O(m n^{-1})).$$

 To enhance the clarity of the proof, we introduce the parameter \(s_n = |\alpha_n - \alpha|^{-1} \wedge n\) and select the following:
 \[
 \ell_{1,\alpha}(n) = \sqrt{\frac{1}{10} \log s_n}, \quad \ell_{2,\alpha}(n) = \frac{4\sqrt{\log s_n} - a_n}{b_n}, \quad j_n = \lfloor s_n^{1/10} \rfloor.
 \]
 The justification for these specific choices will be provided later in the text.

 Given that \(a_n\), \(b_n\), and \(\alpha_n\) are bounded, we observe that \(u_n(m, x) \gg 1\) whenever \(m \ge j_n\) and $x\in[-\ell_{1, \alpha}(n),\ell_{1, \alpha}(n)]$. Therefore, the asymptotic behavior of \(c_n(m, x)\) in this regime mirrors the case where \(\alpha = +\infty\). In contrast, for \(0 \le m \le j_n - 1\) and finite $x$, the quantity \(u_n(m, x)\) varies from a constant to positive infinity. This range of values introduces significant complexity into obtaining a precise asymptotic description of \(c_n(m, x)\).

 Next, we apply Lemma \ref{ed} to obtain a precise estimate of
 $c_n(m, x)$ including the case $u_n(m, x)$ is bounded. 
 \begin{lem}\label{m1}
 	Let $0\leq m\leq j_n-1.$  Set $$c_1=\frac{\sqrt{\log(\alpha+1)}}{w(\alpha+1)}+\frac{2}{\sqrt{\log(\alpha+e)}w(\alpha+e^{\frac{1}{\sqrt{2\pi}}})}-\frac{\log(\sqrt{2\pi}\log(\alpha+e^{\frac{1}{\sqrt{2\pi}}}))}{w(\alpha+e)\sqrt{\log(\alpha+e)}}$$
and
$$c_2=\frac{1}{2(\alpha+e)(\log(\alpha+e))^{3/2}},$$
where $w(t)=2t\log t.$
Then, we have 
 	$$c_n(m, x)=1-\Phi(v_\alpha(m, x))-\phi(v_\alpha(m, x))\left(\frac{1}{n}q_{1}(m, x)+(\alpha_n-\alpha)q_2(m, x)\right)+O(s_n^{-3/2})$$  uniformly on  $-\ell_{1,\alpha}(n)\leq x\leq \ell_{2,\alpha}(n)$, 
 	where 
  	 $$q_{1}(m, x)=\frac{2\alpha(v^{2}_\alpha(m, x)-1)-3\sqrt{\alpha}(2m+1)v_{\alpha}(m, x)+6m(m+1)}{12\sqrt{\alpha}}$$ \text{and}$$ q_{2}(m, x)= c_1-c_2x-\frac{m}{2\alpha^{3/2}}.$$ 
 \end{lem}
 \begin{proof} 
 By Lemma \ref{ed}, we obtain
\begin{equation}\label{ccc}
c_n(m,x) = 1 - \Phi(v_n(m, x)) + \frac{1 - v_n^2(m, x)}{6\sqrt{n k_n}} \phi(v_n(m, x)) + O(n^{-3/2}).
\end{equation}

We next reduce the expression for \(v_n(m, x)\) to \(v_{\alpha}(m, x)\). Recall that
\[
a_n = \sqrt{\log(\alpha_n + 1)} - \frac{\log\left(\sqrt{2\pi} \log\left(\alpha_n + e^{\frac{1}{\sqrt{2\pi}}}\right)\right)}{\sqrt{\log(\alpha_n + e)}}, \quad
b_n = \frac{1}{\sqrt{\log(\alpha_n + e)}}.
\]
Using Taylor expansion and the fact that \(\alpha_n - \alpha = o(1)\), we have for any fixed \(r > 0\),
\[
\sqrt{\log(\alpha_n + r)} = \sqrt{\log(\alpha + r)} + \frac{\alpha_n - \alpha}{2(\alpha + r)\sqrt{\log(\alpha + r)}} + O((\alpha_n - \alpha)^2),
\]
\[
\frac{1}{\sqrt{\log(\alpha_n + r)}} = \frac{1}{\sqrt{\log(\alpha + r)}} - \frac{\alpha_n - \alpha}{2(\alpha + r)(\log(\alpha + r))^{3/2}} + O((\alpha_n - \alpha)^2),
\]
as well as
\[
a_n = a + c_1(\alpha_n - \alpha) + O((\alpha_n - \alpha)^2), \quad
b_n = b - c_2(\alpha_n - \alpha) + O((\alpha_n - \alpha)^2).
\]

Similarly,
\[
\frac{m}{\sqrt{\alpha_n}} = \frac{m}{\sqrt{\alpha}} (1 - \frac{\alpha_n - \alpha}{2\alpha})\left(1 + O((\alpha_n - \alpha)^2)\right),
\]
and by Lemma \ref{diagammapro},
\[
\frac{1}{\sqrt{n \psi'(n - m)}} = 1 - \frac{2m + 1}{4n} + O(\frac{m^2}{n^2}),
\]
\[
\psi(n) - \psi(n - m) = \frac{m}{n} + \frac{m(m + 1)}{2n^2} + O(\frac{m^3}{n^3}).
\]

Therefore,
\begin{equation}\label{vneq1}
\begin{aligned}
&\frac{k_n (\psi(n) - \psi(n - m))}{\sqrt{k_n \psi'(n - m)}} \\
&= \frac{m}{\sqrt{\alpha_n}} (1 - \frac{2m + 1}{4n} + \frac{m + 1}{2n})(1 + O(\frac{m^2}{n^2})) \\
&= \frac{m}{\sqrt{\alpha}} (1 + \frac{1}{4n} - \frac{\alpha_n - \alpha}{2\alpha})(1 + O(m^2 n^{-2} + (\alpha_n - \alpha)^2)),
\end{aligned}
\end{equation}
and
\begin{equation}\label{vneq2}
\begin{aligned}
&\frac{a_n + b_n x}{\sqrt{n \psi'(n - m)}} \\
&= (a + b x - \frac{(a + b x)(2m + 1)}{4n} + (c_1 - c_2 x)(\alpha_n - \alpha)) \\
&\quad \times (1 + O(\frac{m^2}{n^2} + (\alpha_n - \alpha)^2)).
\end{aligned}
\end{equation}

The choices of \(s_n\), \(m\), and \(\ell_{i, \alpha}(n)\) (for \(i = 1, 2\)) ensure that
\[
\frac{m^2 (m + |x|)}{n^2} + (\alpha_n - \alpha)^2 (m + |x|) \lesssim s_n^{-\frac{17}{10}}.
\]

Substituting \eqref{vneq1} and \eqref{vneq2} into \eqref{vndef} yields
\begin{equation}\label{asymforun0}
\begin{aligned}
v_n(m, x) &= v_{\alpha}(m, x) - \frac{(2m + 1)v_{\alpha}(m, x)}{4n} + \frac{m(m + 1)}{2\sqrt{\alpha} \, n} \\
&\quad + (\alpha_n - \alpha)(c_1 - c_2 x - \frac{m}{2\alpha^{3/2}}) + O(s_n^{-\frac{17}{10}}).
\end{aligned}
\end{equation}

Define
\[
g_n(m, x) := -\frac{(2m + 1)v_{\alpha}(m, x)}{4n} + \frac{m(m + 1)}{2\sqrt{\alpha} \, n} + (\alpha_n - \alpha)\left(c_1 - c_2 x - \frac{m}{2\alpha^{3/2}}\right) + O(s_n^{-\frac{17}{10}}),
\]
so that \(g_n(m, x) \lesssim s_n^{-\frac45}\). 

Applying Taylor expansions to \(\Phi\) and \(\phi\), we obtain
\[
\Phi(v_n(m, x)) = \Phi(v_\alpha(m, x)) + \phi(v_\alpha(m, x)) \left(g_n(m, x) + O(v_{\alpha}(m, x) g_n^2(m, x))\right),
\]
and
\[
\phi(v_n(m, x)) = \phi(v_\alpha(m, x)) \left(1 + O(v_{\alpha}(m, x) g_n(m, x))\right).
\]

 Note that for $0\leq m\leq s_n^{\frac1{10}}$ and $-\ell_{1,\alpha}(n)\leq x\leq \ell_{2,\alpha}(n), $ 
 	$$v_{\alpha}(m, x)g_n^{2}(m,x)\lesssim s_n^{-\frac32} \quad \text{ and} \quad \frac{v_\alpha(m,x)g_n(m,x)}{n}\lesssim s_n^{-\frac{17}{10}}.$$
 	Therefore, it follows from \eqref{ccc} that 
 			$$\aligned 
 			&c_n(m,x)\\
 			=&1-\Phi(v_\alpha(m, x))-\phi(v_\alpha(m, x))\left(\frac{\sqrt{\alpha}(v^{2}_\alpha(m, x)-1)}{6n}+g_n(m,x)+O(s_n^{-\frac{3}{2}})\right)+O(n^{-\frac{3}{2}})\\
 			=&1-\Phi(v_\alpha(m, x))-\phi(v_\alpha(m, x))(n^{-1} q_{1}(m, x)+(\alpha_n-\alpha)q_2(m, x))+O(s_n^{-\frac32})
 			.\endaligned $$
 	
 			The proof is then completed. 
 \end{proof} 
 
 Next, we work on the case where $m\ge j_n.$
 \begin{lem}\label{jn1}
 We have
 	\begin{equation}\label{tail} \sum_{m=j_n}^{n-1}\log(1-c_n(m,x))=o(e^{-\frac{ s_n^{1/5}}{3\alpha}})\quad \text{and}\quad
 	\sum_{m=j_n}^{+\infty}\log\Phi(v_\alpha(m,x))=o(e^{-\frac{ s_n^{1/5}}{3\alpha}}) \end{equation}
 	uniformly on $x\ge -\ell_{1, \alpha}(n)$ as $n\to\infty.$
 \end{lem}
 \begin{proof} Since both \(1 - c_n(m, x)\) and \(\Phi(v_{\alpha}(m, x))\) are increasing in \(x\) and bounded above by 1, it suffices to prove \eqref{tail} for \(x = -\ell_{1,\alpha}(n)\).

By Lemma \ref{diagammapro}, we have
\[
v_n(m, -\ell_{1,\alpha}(n)) = u_n(m, -\ell_{1,\alpha}(n))(1 + O(m n^{-1})) = O(m).
\]
Applying the central limit theorem to the sequence \((\log S_{n - j_n, r})_{1 \le r \le k_n}\), we obtain
\begin{equation}\label{c1}
c_n(m, -\ell_{1,\alpha}(n)) = \frac{1 + o(1)}{\sqrt{2\pi} \, u_n(m, -\ell_{1,\alpha}(n))} \exp\left(-\frac{1}{2} u_n^2(m, -\ell_{1,\alpha}(n))\right)
\end{equation}
uniformly for \(m \le n^{1/10}\).

The monotonicity of \(c_n(m, x)\) in \(m\) implies
\[
c_n(m, -\ell_{1,\alpha}(n)) \le c_n(j_n, -\ell_{1,\alpha}(n)) \lesssim \frac{1}{u_n(j_n, -\ell_{1,\alpha}(n))} \exp\left(-\frac{1}{2} u_n^2(j_n, -\ell_{1,\alpha}(n))\right) = o(1)
\]
uniformly for \(m \ge j_n\), noting that
\begin{equation}\label{256}
u_n(j_n, -\ell_{1,\alpha}(n)) = \frac{j_n}{\sqrt{\alpha}}(1 + o(1)) \gg 1.
\end{equation}
Consequently,
\begin{equation}\label{cm}
\sum_{m = j_n}^{n - 1} \log(1 - c_n(m, -\ell_{1,\alpha}(n))) = -\sum_{m = j_n}^{n - 1} c_n(m, -\ell_{1,\alpha}(n))(1 + o(1)).
\end{equation}

We now consider two cases.

Case 1: \((\alpha_n - \alpha)^{-1} \ll n\), i.e., \(j_n = \lfloor (\alpha_n - \alpha)^{-\frac1{10}} \rfloor\). Let \(w_n = \lfloor n^{\frac1{10}} \rfloor\). Then by \eqref{c1},
\[
\begin{aligned}
&\sum_{m = j_n}^{n - 1} c_n(m, -\ell_{1,\alpha}(n)) \\
&= \sum_{m = j_n}^{w_n} c_n(m, -\ell_{1,\alpha}(n)) + \sum_{m = w_n + 1}^{n - 1} c_n(m, -\ell_{1,\alpha}(n)) \\
&\lesssim \sum_{m = j_n}^{w_n} \frac{\exp\left(-\frac{1}{2} u_n^2(m, -\ell_{1,\alpha}(n))\right)}{u_n(m, -\ell_{1,\alpha}(n))} + \frac{n \exp\left(-\frac{1}{2} u_n^2(w_n, -\ell_{1,\alpha}(n))\right)}{u_n(w_n, -\ell_{1,\alpha}(n))}.
\end{aligned}
\]
Applying Lemma \ref{sum} with bounded \(\gamma_n = \sqrt{\alpha_n}\) to the first sum yields
\[
\sum_{m = j_n}^{n - 1} c_n(m, -\ell_{1,\alpha}(n)) \lesssim \frac{\exp\left(-\frac{1}{2} u_n^2(j_n, -\ell_{1,\alpha}(n))\right)}{u_n(j_n, -\ell_{1,\alpha}(n))} + \frac{n \exp\left(-\frac{1}{2} u_n^2(w_n, -\ell_{1,\alpha}(n))\right)}{u_n(w_n, -\ell_{1,\alpha}(n))}.
\]
Note that
\[
u_n(w_n, -\ell_{1,\alpha}(n)) = \frac{n^{\frac1{10}}}{\sqrt{\alpha}}(1 + o(1)),
\]
which implies
\[
\frac{n \exp\left(-\frac{1}{2} u_n^2(w_n, -\ell_{1,\alpha}(n))\right)}{u_n(w_n, -\ell_{1,\alpha}(n))} \lesssim \exp\{ -\frac{n^{\frac1{5}}(1 + o(1))}{2\alpha} + \frac{9}{10} \log n\} \ll \exp\{ -\frac{n^{\frac1{5}}}{3\alpha} \}.
\]
It follows from \eqref{256} that
\[
\frac{\exp\left(-\frac{1}{2} u_n^2(j_n, -\ell_{1,\alpha}(n))\right)}{u_n(j_n, -\ell_{1,\alpha}(n))} + \frac{n \exp\left(-\frac{1}{2} u_n^2(w_n, -\ell_{1,\alpha}(n))\right)}{u_n(w_n, -\ell_{1,\alpha}(n))} = o(\exp(-\frac{1}{3\alpha} n^{\frac15})).
\]

Case 2: \((\alpha_n - \alpha)^{-1} \gtrsim n\), under which  \(j_n = w_n\). Then
\[
\sum_{m = j_n}^{n - 1} c_n(m, -\ell_{1,\alpha}(n))) = \sum_{m = w_n}^{n - 1} c_n(m, -\ell_{1,\alpha}(n))) \lesssim \frac{n \exp\left(-\frac{1}{2} u_n^2(w_n, -\ell_{1,\alpha}(n))\right)}{u_n(w_n, -\ell_{1,\alpha}(n))} \ll e^{-\frac{j_n^2}{3\alpha}}.
\]
In both cases, we conclude
\[
\sum_{m = j_n}^{n - 1} c_n(m, -\ell_{1,\alpha}(n))) =o(\exp(-\frac{s_n^{\frac15}}{3\alpha})),
\]
and hence from \eqref{cm},
\[
\sum_{m = j_n}^{n - 1} \log(1 - c_n(m, -\ell_{1,\alpha}(n))) = o(\exp(-\frac{s_n^{\frac15}}{3\alpha})).
\]

Next, we establish the second asymptotic relation in \eqref{tail}. For  \(m \geq j_n\), we have
\[v_\alpha(m, -\ell_{1,\alpha}(n))\geq
v_\alpha(j_n, -\ell_{1,\alpha}(n)) \gg 1.
\]
By Mills' ratio, it follows that
\[
1 - \Phi(v_\alpha(m, -\ell_{1,\alpha}(n))) = \frac{1 + o(1)}{\sqrt{2\pi}}  \frac{\exp\left(-\frac{1}{2} v^2_\alpha(m, -\ell_{1,\alpha}(n))\right)}{v_\alpha(m, -\ell_{1,\alpha}(n))}.
\]
Consequently,
\[
\sum_{m = j_n}^{+\infty} \log \Phi(v_\alpha(m, -\ell_{1,\alpha}(n))) = -\frac{1 + o(1)}{\sqrt{2\pi}} \sum_{m = j_n}^{+\infty} \frac{\exp\left(-\frac{1}{2} v^2_\alpha(m, -\ell_{1,\alpha}(n))\right)}{v_\alpha(m, -\ell_{1,\alpha}(n))}.
\]

Applying Lemma \ref{sum} once more, and using the estimate
\[
v_\alpha(j_n, -\ell_{1,\alpha}(n)) = \frac{j_n}{\sqrt{\alpha}}(1 + o(1)),
\]
we obtain
\[
\sum_{m = j_n}^{+\infty} \frac{\exp\left(-\frac{1}{2} v^2_\alpha(m, -\ell_{1,\alpha}(n))\right)}{v_\alpha(m, -\ell_{1,\alpha}(n))} \lesssim \frac{\exp\left(-\frac{1}{2} v^2_\alpha(j_n, -\ell_{1,\alpha}(n))\right)}{v_\alpha(j_n, -\ell_{1,\alpha}(n))} \ll \exp\left(-\frac{j_n^2}{3\alpha}\right).
\]
Hence,
\[
\sum_{m = j_n}^{+\infty} \log \Phi(v_\alpha(m, -\ell_{1,\alpha}(n))) = o\left(\exp\left(-\frac{s_n^{1/5}}{3\alpha}\right)\right).
\]
This completes the proof.	 	 	
 \end{proof}
 
 With Lemmas \ref{m1} and \ref{jn1} in hand, we now proceed to the proof of Theorem \ref{main} for $\alpha \in (0, +\infty)$.
 
  \begin{proof}[{\bf Proof of Theorem \ref{main} for $\alpha\in(0,+\infty)$}]
  	Recall that \[\beta_n(x) := \sum_{m=0}^{n-1} \log(1 - c_n(m, x)).\] By Lemma \ref{jn1}, we obtain
\[
\begin{aligned}
\beta_n(x) - \sum_{m=0}^{+\infty} \log \Phi(v_{\alpha}(m, x))
&= \sum_{m=0}^{j_n - 1} \left( \log(1 - c_n(m, x)) - \log \Phi(v_{\alpha}(m, x)) \right) + o\left(e^{-\frac{s_n^{1/5}}{3\alpha}}\right) \\
&= \sum_{m=0}^{j_n - 1} \log\left(1 + \frac{1 - c_{n}(m, x) - \Phi(v_{\alpha}(m, x))}{\Phi(v_{\alpha}(m, x))}\right) + o\left(e^{-\frac{s_n^{1/5}}{3\alpha}}\right).
\end{aligned}
\]
From Lemma \ref{m1}, we have
\[
\begin{aligned}
d_\alpha(m,x) :&= 1 - c_{n}(m, x) - \Phi(v_{\alpha}(m, x)) \\
&= \phi(v_\alpha(m, x)) \left( n^{-1} q_1(m, x) + (\alpha_n - \alpha) q_2(m, x) \right) + O(s_n^{-3/2}).
\end{aligned}
\]
The monotonicity of \(\Phi\) implies
\[
\frac{1}{\Phi(v_\alpha(m,x))} \leq \frac{1}{\Phi(v_\alpha(0, -\ell_{1,\alpha}(n)))} \lesssim v_\alpha(0, -\ell_{1,\alpha}(n)) e^{\frac{v^2_\alpha(0, -\ell_{1,\alpha}(n))}{2}} \lesssim \sqrt{\log s_n} \, s_n^{1/20}
\]
uniformly for \(-\ell_{1,\alpha}(n) \leq x \leq \ell_{2,\alpha}(n)\) and \(0 \leq m \leq j_n - 1\).
Thus, we derive
\[
\frac{|d_\alpha(m,x)|}{\Phi(v_{\alpha}(m, x))} \lesssim s_n^{-\frac34} \sqrt{\log s_n},
\]
and
$$ \sum_{m=0}^{j_n - 1} \frac{|d_\alpha(m,x)|}{\Phi(v_{\alpha}(m, x))}\gtrsim\frac{|d_\alpha(0,x)|}{\Phi(v_{\alpha}(0, x))}\gtrsim s_n^{-\frac{19}{20}}\sqrt{\log s_n}\gtrsim e^{-\frac{s_n^{\frac15}}{3\alpha}}. $$
Hence
\[
| \beta_n(x) - \sum_{m=0}^{+\infty} \log \Phi(v_{\alpha}(m, x)) | =(1+o(1)\sum_{m=0}^{j_n - 1} \frac{|d_\alpha(m,x)|}{\Phi(v_{\alpha}(m, x))}$$ and $$ | \beta_n(x) - \sum_{m=0}^{+\infty} \log \Phi(v_{\alpha}(m, x)) |\lesssim j_n s_n^{-\frac34} \sqrt{\log s_n} = s_n^{-\frac{13}{20}} \sqrt{\log s_n}=o(1)
\]
for \(-\ell_{1,\alpha}(n) \leq x \leq \ell_{2,\alpha}(n)\) and \(0 \leq m \leq j_n - 1\).

It follows that
\begin{equation}\label{basicabove}
\begin{aligned}
&\left| \mathbb{P}(X_n \le x) - \Phi_{\alpha}(x) \right| \\
&= \Phi_{\alpha}(x) \left| \exp\left( \beta_n(x) - \sum_{m=0}^{+\infty} \log \Phi(v_{\alpha}(m, x)) \right) - 1 \right| \\
&= (1 + o(1)) \, \Phi_{\alpha}(x) \left| \sum_{m=0}^{+\infty} \frac{d_{\alpha}(m, x)}{\Phi(v_{\alpha}(m, x))} \right| \\
&= (1 + o(1)) \, \Phi_{\alpha}(x) \left| \sum_{m=0}^{+\infty} \frac{\phi(v_{\alpha}(m, x))}{\Phi(v_{\alpha}(m, x))} \left( n^{-1} q_1(m, x) + (\alpha_n - \alpha) q_2(m, x) \right) \right|.
\end{aligned}
\end{equation}

As shown in the next section,
\[
\sup_{x \in \mathbb{R}} \Phi_\alpha(x) \sum_{m=0}^{\infty} |q_k(m, x)| \frac{\phi(v_\alpha(m,x))}{\Phi(v_\alpha(m,x))} < +\infty.
\]

Letting
\[
\eta := \lim_{n \to \infty} (\alpha_n - \alpha)n,
\]
we identify the dominant order of the error term according to the value of \(\eta\).

	\begin{itemize}
    \item When \(|\eta| = +\infty\), the dominant term is \(\alpha_n - \alpha\), and we have
    \[
    \begin{aligned}
    &\sup_{-\ell_{1,\alpha}(n) \le x \le \ell_{2,\alpha}(n)} \left| \mathbb{P}(X_n \le x) - \Phi_\alpha(x) \right| \\
    &= |\alpha_n - \alpha| \sup_{x \in \mathbb{R}} \Phi_\alpha(x) \left| \sum_{m=0}^{\infty} q_{2,m}(x) \frac{\phi(v_\alpha(m,x))}{\Phi(v_\alpha(m,x))} \right| (1 + o(1)).
    \end{aligned}
    \]

    \item When \(\eta = 0\),
    \[
    \begin{aligned}
    &\sup_{-\ell_{1,\alpha}(n) \le x \le \ell_{2,\alpha}(n)} \left| \mathbb{P}(X_n \le x) - \Phi_\alpha(x) \right| \\
    &= \frac{1 + o(1)}{n} \sup_{x \in \mathbb{R}} \Phi_\alpha(x) \left| \sum_{m=0}^{\infty} q_{1,m}(x) \frac{\phi(v_\alpha(m,x))}{\Phi(v_\alpha(m,x))} \right|.
    \end{aligned}
    \]

    \item For the case \(|\eta| > 0\),
    \[
    \begin{aligned}
    &\sup_{-\ell_{1,\alpha}(n) \le x \le \ell_{2,\alpha}(n)} \left| \mathbb{P}(X_n \le x) - \Phi_\alpha(x) \right| \\
    &= \frac{1 + o(1)}{n} \sup_{x \in \mathbb{R}} \Phi_\alpha(x) \left| \sum_{m=0}^{\infty} \left( q_{1,m}(x) + \eta q_{2,m}(x) \right) \frac{\phi(v_\alpha(m,x))}{\Phi(v_\alpha(m,x))} \right|.
    \end{aligned}
    \]
\end{itemize}

The proof will be completed once we show that the supremum over the remaining regions of the real line is negligible compared to \(s_n^{-1}\).

From \eqref{basicabove}, for \(-\ell_{1,\alpha}(n) \le x \le \ell_{2,\alpha}(n)\), we have
\[
\left| \mathbb{P}(X_n \le x) - \Phi_\alpha(x) \right| = \Phi_\alpha(x) \cdot O(s_n^{-1}).
\]
By the monotonicity of \(\Phi_\alpha\),
\[
\sup_{x \in (-\infty, -\ell_{1,\alpha}(n)]} \left| \mathbb{P}(X_n \le x) - \Phi_\alpha(x) \right| \le \mathbb{P}(X_n \le -\ell_{1,\alpha}(n)) + \Phi_\alpha(-\ell_{1,\alpha}(n)) \lesssim \Phi_\alpha(-\ell_{1,\alpha}(n)),
\]
and similarly,
\[
\sup_{x \in [\ell_{2,\alpha}(n), +\infty)} \left| \mathbb{P}(X_n \le x) - \Phi_\alpha(x) \right| \lesssim 1 - \Phi_\alpha(\ell_{2,\alpha}(n)).
\]
Choose \(m_1 = \lfloor \sqrt{\alpha} (b \ell_{1,\alpha}(n)/2 - a) \rfloor\) such that
\[
v_\alpha(m_1, x) = -\frac{b \ell_{1,\alpha}(n)}{2}(1 + o(1)).
\]
Then,
\[
\Phi_\alpha(-\ell_{1,\alpha}(n)) \le \left( \Phi(v_\alpha(m_1, -\ell_{1,\alpha}(n))) \right)^{m_1} \ll \ell_{1,\alpha}^{-m_1}(n) \exp\left( -\frac{\sqrt{\alpha} b^3 \ell_{1,\alpha}^3(n)}{20} \right) \ll s_n^{-1}.
\]
The choice \(\ell_{2,\alpha}(n) = \frac{4\sqrt{\log s_n} - a_n}{b_n}\) ensures \(v_\alpha(0, \ell_{2,\alpha}(n)) = 4\sqrt{\log s_n} \gg 1\), so that
\[
\begin{aligned}
1 - \Phi_\alpha(\ell_{2,\alpha}(n))
&= 1 - \exp\left\{ \sum_{j=0}^{+\infty} \log \Phi(v_\alpha(j, \ell_{2,\alpha}(n))) \right\} \\
&= 1 - \exp\left\{ -\sum_{j=0}^{+\infty} \frac{1 + o(1)}{\sqrt{2\pi}} v_\alpha^{-1}(j, \ell_{2,\alpha}(n)) \exp\left( -\frac{v_\alpha^2(j, \ell_{2,\alpha}(n))}{2} \right) \right\} \\
&\le \sum_{j=0}^{+\infty} v_\alpha^{-1}(j, \ell_{2,\alpha}(n)) \exp\left( -\frac{v_\alpha^2(j, \ell_{2,\alpha}(n))}{2} \right) \\
&\lesssim v_\alpha^{-1}(0, \ell_{2,\alpha}(n)) \exp\left( -\frac{v_\alpha^2(0, \ell_{2,\alpha}(n))}{2} \right) = 16 s_n^{-7} (\log s_n)^{-1/2} \ll s_n^{-1}.
\end{aligned}
\]
Therefore,
\[
\sup_{x \in \mathbb{R}} \left| \mathbb{P}(X_n \le x) - \Phi(x) \right| = \sup_{-\ell_{1,\alpha}(n) \le x \le \ell_{2,\alpha}(n)} \left| \mathbb{P}(X_n \le x) - \Phi(x) \right|.
\]
This completes the proof of Theorem \ref{main} for the case \(\alpha \in (0, +\infty)\).

  \end{proof}
  \section{The case $ \alpha=0.$}
  In this section, we address the case \(\alpha = 0\). Recall that for \(0 \leq m \leq n-1\),
\[
c_n(m, x) = \mathbb{P}\left( \log Y_{n - m} > k_n \psi(n) + \alpha_n^{-1/2}(a_n + b_n x) \right),
\]
and define
\[
u_n(m, x) = \frac{m}{\sqrt{\alpha_n}} + a_n + b_n x.
\]
Let \(z_n = \alpha_n^{-1/2} \wedge n\), and choose \(\ell_{2,0}(n) = \ell_{1,0}(n) = z_n^{1/10}\). We begin the proof by deriving a precise estimate for \(c_n(m, x)\).

Since \(k_n \gg 1\), we apply an Edgeworth expansion to \(\log Y_n\) to obtain an asymptotic expression for \(c_n(0, x)\) that holds uniformly for \(x \in [-\ell_{1,0}(n), \ell_{1,0}(n)]\). Furthermore, we will show that
\[
\mathbb{P}(X_n \leq x) = \prod_{m=0}^{n-1} (1 - c_n(m, x)) = (1 - c_n(0, x))(1 + o(1)).
\]

The following lemma provides a precise estimate for \(c_n(0, x)\) using the Edgeworth expansion.
\begin{lem}\label{ed2}
    	For $-\ell_{1,0}(n)\leq x\leq \ell_{1,0}(n),$ we have 
  	$$c_n(0,x)=1-\Phi(x)-(\sqrt{\alpha_n}-\frac{x}{4n})\phi(x)+O(z_n^{-19/10}).$$
    \end{lem}
    \begin{proof}
    We note that
\[
\begin{aligned}
c_n(0, x) &= \mathbb{P}\left( \log Y_n > k_n \psi(n) + (k_n / n)^{1/2} (a_n + b_n x) \right) \\
&= \mathbb{P}\left( \frac{\log Y_n - k_n \psi(n)}{\sqrt{k_n \psi'(n)}} > (n \psi'(n))^{-1/2} (a_n + b_n x) \right).
\end{aligned}
\]

Under the condition \(\alpha = 0\), the parameters \(a_n\) and \(b_n\) satisfy the asymptotic relations
\begin{equation}\label{alpha0}
a_n = \sqrt{\alpha_n} + O(\alpha_n), \quad \text{and} \quad b_n = 1 + O(\alpha_n).
\end{equation}
From Lemma \ref{diagammapro}, we obtain
\[
\begin{aligned}
&(n \psi'(n))^{-1/2} (a_n + b_n x) \\
&= \left(1 - \frac{1}{4n} + O(n^{-2})\right) \left( \sqrt{\alpha_n} + x + O(\alpha_n (|x| + 1)) \right) \\
&= x + \sqrt{\alpha_n} - \frac{x}{4n} + O(z_n^{-19/10})
\end{aligned}
\]
uniformly for \(-\ell_{1,0}(n) \le x \le \ell_{1,0}(n)\).
Define
\[
y_n(x) = \sqrt{\alpha_n} - \frac{x}{4n} + O(z_n^{-19/10}) = O(z_n^{-1}) \quad \text{and} \quad t_n(x) = x + y_n(x).
\]
By Lemma \ref{ed}, it follows that
\[
1 - c_n(0, x) = \Phi(t_n(x)) - \frac{1 - t_n^2(x)}{6\sqrt{k_n n}} \phi(t_n(x)) + O\left(k_n^{-3/2} + k_n^{-1/2} n^{-3/2}\right).
\]
Since \(x^2 \phi(x)\) is bounded, and noting that
\[
\frac{1}{\sqrt{k_n n}} = \sqrt{\alpha_n} n^{-1} \lesssim z_n^{-2}, \quad k_n^{-3/2} + k_n^{-1/2} n^{-3/2} = (\alpha_n^{-1} n)^{-3/2} + \sqrt{\alpha_n} n^{-2} \lesssim z_n^{-3},
\]
we conclude that
\[
\frac{1 - t_n^2(x)}{6\sqrt{k_n n}} \phi(t_n(x)) + O\left(k_n^{-3/2} + k_n^{-1/2} n^{-3/2}\right) = O(z_n^{-2}).
\]
For \(-\ell_{1,0}(n) \le x \le \ell_{1,0}(n)\), a Taylor expansion of \(\Phi(t_n(x))\) about \(x\) gives
\[
\Phi(t_n(x)) = \Phi(x) + \phi(x) \left( y_n(x) + O(y_n^2(x) x) \right) = \Phi(x) + \left( \sqrt{\alpha_n} - \frac{x}{4n} \right) \phi(x) + O(z_n^{-19/10}).
\]
Therefore,
\[
c_n(0, x) = 1 - \Phi(x) - \left( \sqrt{\alpha_n} - \frac{x}{4n} \right) \phi(x) + O(z_n^{-19/10}).
\] 
The proof is completed. 
    \end{proof}
    
   Next, we demonstrate that the product $\prod_{m=1}^{n-1} (1 - c_n(m, x))$ is negligible compared to $1 - c_n(0, x)$.

    \begin{lem}\label{88}
    	 For $-\ell_{1,0}(n)\leq x\leq\ell_{1,0}(n),$ we have
    	 	$$\prod_{m=1}^{n-1}(1-c_n(m,x))=1+O(z_n^{-9/2}).$$
    \end{lem}
    \begin{proof}
    	Since $k_n\gg 1$ and $u_n(m,x)\gg1$  	for $1\leq m\leq \sqrt{n}$ and $x\geq -\ell_{1,0}(n),$ we get from Berry-Essen bound for the sum of i.i.d random sequence (\cite{Chen2011}) that  
    	$$c_n(m,x)\leq \frac{1}{u_n(m,x)}e^{-\frac{u_n^{2}(m,x)}{3}}+O(k_n^{-\frac{1}{2}}u^{-3}_n(m,x)).$$ 
    	It follows from the monotonicity of $c_n(m, x)$ that
    		$$\aligned c_n(m,x)&\leq c_n(1,-\ell_{1,0}(n))\\
    	&\leq\frac{1}{u_n(1,-\ell_{1,0}(n))}e^{-\frac{u^{2}_n(1,-\ell_{1,0}(n))}{3}}+O(k_n^{-\frac{1}{2}}u^{-3}_n(1,-\ell_{1,0}(n)))\ll1,\endaligned$$
    	uniformly on  $-\ell_{1,0}(n)\leq x\leq\ell_{1,0}(n)$ and $1\leq m\leq n-1.$ Hence,
    	\begin{equation}\label{511}
    		\prod_{m=1}^{n-1}(1-c_n(m,x))=\exp\{-(1+o(1))\sum_{m=1}^{n-1}c_n(m,x)\}.
    	\end{equation}
    	Next, we prove that $\sum_{m=1}^{n-1}c_n(m,x)=o(1)$ using summation formulas. Note that 
    $u_n(1,x)=\alpha_n^{-1/2}(1+o(1)),$ which implies 
    \begin{equation}\label{67}
    \sum_{m=1}^{+\infty}\frac{1}{u_n(m,x)}e^{-\frac{u_n^{2}(m,x)}{3}}\lesssim\frac{1}{u_n(1,x)}e^{-\frac{u_n^{2}(1,x)}{3}}\quad
    \text{and}\quad
    \sum_{m=1}^{+\infty}u^{-3}_n(m,x)\lesssim u_{n}^{-3}(1,x).
    \end{equation}
   	  We divide the sum into two parts, 
   	$$ \sum_{m=1}^{n-1}c_n(m,x)\leq\sum_{m=1}^{[n^{5/6}]}c_n(m,x)+nc_n([n^{5/6}],x),$$ and estimate an upper bound for each segment.
   	From the equation \eqref{67}, we know that
   	$$\aligned \sum_{m=1}^{[n^{5/6}]}c_n(m,x)&\leq\sum_{m=1}^{n^{5/6}}\left(u_n^{-1}(m,x)e^{-\frac{u_n^{2}(m,x)}{3}}+O(k_n^{-\frac{1}{2}}u^{-3}_n(m,x))\right)\\&\lesssim u_n^{-1}(1,x)e^{-\frac{u_n^{2}(1,x)}{3}}+k_n^{-\frac{1}{2}}u_{n}^{-3}(1,x)\\
   	&\lesssim\sqrt{\alpha_n}e^{-\frac{1}{4\alpha_n}}+k_n^{-\frac{1}{2}}\alpha_n^{\frac{3}{2}}\lesssim z_n^{-9/2}.\endaligned$$
   	 Now $u_{n}([n^{5/6}],x)\geq u_{n}([n^{5/6}],-\ell_{1,0}(n))=n^{1/3}\sqrt{k_n}(1+o(1)),$ it follows that
   	$$\aligned nc_n([n^{5/6}],x)&\leq\frac{n}{u_n([n^{5/6}],-\ell_{1,0}(n))}e^{-\frac{u_n^{2}([n^{5/6}], -\ell_{1,0}(n))}{3}}+O(nk_n^{-\frac{1}{2}}u^{-3}_n([n^{5/6}], -\ell_{1,0}(n)))\\
   	&\lesssim \frac{1}{k_n^{2}}\lesssim z_n^{-6}.\endaligned $$
   	 Consequently, $\sum_{m=1}^{n-1}c_n(m,x)=O(z_n^{-9/2})$ and, by \eqref{511},
   	$$\prod_{m=1}^{n-1}(1-c_n(m,x))=1+O(z_n^{-9/2}).$$ This closes the proof. 
   	
    \end{proof}
    
    Once we completely understood the behaviors of $c_n(m, x),$ we are able to prove Theorem \ref{main} now. 
     \begin{proof}[{\bf Proof of Theorem \ref{main} for $\alpha=0$}]
     Thanks to  Lemma \ref{ed2} and \ref{88}, we have
     	\begin{equation}\label{79}
     		\aligned
     		|\mathbb{P}\left(X_n\leq x\right)-\Phi(x)|&=|(1-c_n(0,x))(1+O(z_n^{-9/2}))-\Phi(x)|\\&=|1-c_n(0,x)-\Phi(x)+O(z_n^{-9/2})|\\
     		&=|(\sqrt{\alpha_n}-\frac{x}{4n})\phi(x)+O(z_n^{-19/10})| \endaligned
     	\end{equation} 
     	uniformly on $[-\ell_{1,0}(n), \ell_{1,0}(n)].
     	$    	 
     The supremum of $(\sqrt{\alpha_n}-\frac{x}{4n})\phi(x)$ on $\mathbb{R}$ must be attained at some point independent of $n,$ which indicates the supremum would be of order $z_n^{-1}$ and then 
     \begin{equation}\label{middlesup}
     \sup_{|x|\le \ell_{1, 0}(n)}|\mathbb{P}\left(X_n\leq x\right)-\Phi(x)|=(1+o(1))\sup_{x\in\mathbb{R}}\phi(x)|\sqrt{\alpha_n}-\frac{x}{4n}|.
     \end{equation}
It is clear that the dominated order will be either $\sqrt{\alpha_n}$ or $n^{-1}$ and then we define \( \beta := \lim\limits_{n \to \infty} \frac{n^3}{k_n} \) to characterize the dominated order by the value of $\beta.$ 
     Now
     	 for  $h_1\ge 0 ,\,h_2\ge 0$ with $h_1^2+h_2^2>0,$
     	 \begin{equation}\label{suprecal} \sup_{x\in\mathbb{R}}|h_1-h_2x|\phi(x)=\frac{h_1 + \sqrt{h_1^2 + 4h_2^2}}{2\sqrt{2\pi}} \exp\{ \frac{h_1}{h_1+\sqrt{h_1^2 + 4h_2^2}}-\frac12 \}.\end{equation} 
    This expression leads the following estimates. 
     	      	 
     	 - If \( \beta = +\infty \), which means $\sqrt{\alpha_n}\gg n^{-1},$ then  
     	 \[
     	 \sup_{x\in\mathbb{R}}\phi(x)|\sqrt{\alpha_n}-\frac{x}{4n}|= \sqrt{\alpha_n}\sup_{x\in\mathbb{R}}\phi(x)=\frac{\sqrt{\alpha_n}}{\sqrt{2\pi}}.
     	 \]
     	 
     	 - If \( \beta=0\), that is $\sqrt{\alpha_n}\ll n^{-1},$ then  
     	 \[
     	 \sup_{x\in\mathbb{R}}\phi(x)|\sqrt{\alpha_n}-\frac{x}{4n}| = \frac{1 }{4n}\sup_{x\in\mathbb{R}}|x|\phi(x)= \frac{1}{4\sqrt{2\pi e} \, n}.
     	 \]
     	 
     	 - For the case \( \beta \in (0, +\infty) \), we know 
     	 \[\aligned
     	\sup_{x\in\mathbb{R}}\phi(x)|\sqrt{\alpha_n}-\frac{x}{4n}|&= \sqrt{\alpha_n}\sup_{x\in\mathbb{R}}\phi(x)|1-\frac{ x}{ 4n \sqrt{\alpha_n} } |\\
     	&=\frac{2n\sqrt{\alpha_n}+\sqrt{1+4n^2\alpha_n}}{4\sqrt{2\pi} n}\exp(-\frac12(\sqrt{1+4n^2\alpha_n}-2n\sqrt{\alpha_n})^2)\\
     	&=\frac{(2\sqrt{\beta} + \sqrt{1+4\beta})(1+o(1))}{4\sqrt{2\pi e} \, n} \exp\left( \frac{2\sqrt{\beta}}{\sqrt{4\beta + 1} + 2\sqrt{\beta}} \right),\endaligned 
     	 \] 
     	where for the last equality we use  
     	 the fact $n^2\alpha_n=\beta(1+o(1)).$

     Next, we will establish the uniform bound
     	\begin{equation}\label{tailsup}
     		\sup_{|x|>\ell_{1,0}(n)}|\mathbb{P}(X_n\leq x)-\Phi(x)|\ll z_n^{-1}.
     	\end{equation}
     	 Recall
     	$$z_n=\alpha_n^{-1/2}\wedge n\quad\text{and}\quad\ell_{1,0}(n)=z_n^{1/10}, $$
     	which implies $$(\sqrt{\alpha_n}-\frac{\pm\ell_{1,0}(n)}{4n})\phi(\pm\ell_{1,0}(n))+O(z_n^{-19/10})=O(z_n^{-19/10}). $$
     	Then from the equation \eqref{79}, we have
     	$$\aligned \mathbb{P}(X_n\leq -\ell_{1,0}(n))&= \Phi(-\ell_{1,0}(n))+O(z_n^{-19/10}); \\
     	\mathbb{P}(X_n\geq \ell_{1,0}(n))&= 1-\Phi(\ell_{1,0}(n))+O(z_n^{-19/10}).
     	\endaligned $$
     	 By the triangle inequality and the monotonicity of distribution functions, we have
     	 \begin{equation}\label{01}
     	 \aligned\sup_{x<-\ell_{1,0}(n)}|\mathbb{P}(X_n\leq x)-\Phi(x)|&\leq \mathbb{P}(X_n\leq -\ell_{1,0}(n))+\Phi(-\ell_{1,0}(n))\\
     	 &\lesssim \Phi(-\ell_{1,0}(n))+O(z_n^{-19/10})\endaligned
     	 \end{equation}
     	and
     	\begin{equation}\label{02}
     		\aligned\sup_{x>\ell_{1,0}(n)}|\mathbb{P}(X_n\leq x)-\Phi(x)|&\leq \mathbb{P}(X_n\geq \ell_{1,0}(n))+1-\Phi(\ell_{1,0}(n))\\
     		&\lesssim 1-\Phi(\ell_{1,0}(n))+O(z_n^{-19/10}).
     		\endaligned
     	\end{equation}
     	 Note that
     	 \begin{equation}\label{03}
     	 	\Phi(-\ell_{1,0}(n))=1-\Phi(\ell_{1,0}(n))\lesssim \frac{1}{\ell_{1,0}(n)}e^{-\frac{\ell^{2}_{1,0}(n)}{2}}= z_n^{-1/10}e^{-\frac{z_n^{1/5}}{2}}\ll z_n^{-1}.
     	 \end{equation}
     	Having verified equation \eqref{tailsup}, and in view of equation \eqref{middlesup} and the supremum characterization for \(\phi(x)|\sqrt{\alpha_n} - \frac{x}{4n}|\), we conclude the proof of Theorem \ref{main} for \(\alpha = 0\).
     	 
     \end{proof}

 \section{Proof of Remarks}     
 \subsection{Proof of Remark \ref{transition}} Recall 
  $$\Phi_{\alpha}(x)=\prod_{m=0}^{+\infty}\Phi(v_{\alpha}(m, x)).$$ Given that $0<\alpha\ll 1,$ then when $|x|\leq \alpha^{-1/10}$, we have 
  $$v_{\alpha}(m, x)\gg 1, \quad \,\forall m\ge 1.$$ Thus, Mills's ratio again implies 
  $$1-\Phi(v_{\alpha}(m, x))=\frac{1+o(1)}{\sqrt{2\pi} v_{\alpha}(m, x)}e^{-\frac{1}{2}v_{\alpha}^2(m, x)},$$ whence it follows from Lemma \ref{sum} that 
 
  $$\aligned \sum_{m=1}^{+\infty} (1-\Phi(v_{\alpha}(m, x)))
  &\lesssim \sum_{m=1}^{+\infty}\frac{1}{v_{\alpha}(m, x)}\exp(-\frac{1}{2}v_{\alpha}^2(m, x))\\
 & \lesssim \frac{1}{v_{\alpha}(1, x)}\exp(-\frac{1}{2}v_{\alpha}^2(1, x))\\
  &\ll \sqrt{\alpha}e^{-\frac{1}{3\alpha}}.
  \endaligned $$
The last inequality holds because $ v_{\alpha}(1, x)=\alpha^{-1/2}(1+o(1))$ for $|x|\leq \alpha^{-1/10}.$ Thus,
  	\begin{equation}\label{56}
  	\prod_{m=1}^{+\infty}\Phi(v_{\alpha}(m, x))=\exp\{-\sum_{m=1}^{+\infty} (1-\Phi(v_{\alpha}(m, x)))
  	(1+o(1))\}=1+o(\sqrt{\alpha}e^{-\frac{1}{3\alpha}})
  	\end{equation} 
  	 Similarly as \eqref{alpha0}, we have 	 
  $$	 a = \sqrt{\alpha}+O(\alpha)\quad \text{ and} \quad 
  	b = \frac{1}{\sqrt{\log(\alpha + e)}}=1+O(\alpha),$$
 which implies $$v_\alpha(0,x)=a+bx=x+\sqrt{\alpha}+O((|x|+1)\alpha)$$
  	  uniformly on $|x|\leq \alpha^{-1/10}.$ We apply again Taylor's expansion to obtain 
  	  $$\aligned\Phi(v_\alpha(0,x))
  	  &=\Phi(x)+\sqrt{\alpha}\phi(x)+O(\alpha),
  	  \endaligned$$
  	  where the last equality is due to the boundedness of $x\phi(x).$ Taking account of \eqref{56}, we get
  	  $$\Phi_{\alpha}(x)=\Phi(x)+\sqrt{\alpha}\phi(x)+O(\alpha). $$
  	By analogy with equations \eqref{01} and \eqref{02}, we  conclude that
  	 $$\sup_{|x|>\alpha^{-1/10}}|\Phi_{\alpha}(x)-\Phi(x)|\lesssim1-\Phi(\alpha^{-1/10})\lesssim \alpha^{1/10}\exp(-\frac{1}{2\alpha^{1/5}})\ll \sqrt{\alpha}.$$
  	 Therefore,
  	  $$\sup_{x\in\mathbb{R}}|\Phi_{\alpha}(x)-\Phi(x)|=\sup_{x\in\mathbb{R}}\sqrt{\alpha}\phi(x)(1+o(1))=\sqrt{\frac{\alpha}{2\pi}}(1+o(1)).$$
  	  Now consider $\alpha\gg1.$ Since $a\gg 1,$ then for any
  	  $m\geq 0$ and $|x|\leq \frac{1}{2}\log\log\alpha,$ we have
  	  $$v_\alpha(m,x)\gg 1 ,$$ 
  	  which implies
  	 	\begin{equation}
  	 \aligned	\prod_{m=0}^{+\infty}\Phi(v_{\alpha}(m, x))&=\exp\{-(1+o(1))\sum_{m=0}^{+\infty} (1-\Phi(v_{\alpha}(m, x)))
  	 	\}\\
  	 	&= \exp\{-(1+o(1))\sum_{m=0}^{+\infty}\frac{1}{\sqrt{2\pi}v_{\alpha}(m, x)}\exp(-\frac{1}{2}v_{\alpha}^2(m, x))\}.
  	 	\endaligned
  	 \end{equation} 
  	  Lemma \ref{sum} ensures that 
  	  $$\beta_{\alpha}(x):= \sum_{m=0}^{+\infty}\frac{1}{\sqrt{2\pi}v_{\alpha}(m, x)}\exp(-\frac{1}{2}v_{\alpha}^2(m, x))=\frac{\sqrt{\alpha}(1+O((\log  \alpha)^{-1}))}{\sqrt{2\pi}v^{2}_{\alpha}(0, x)}\exp(-\frac{1}{2}v_{\alpha}^2(0, x))$$
  	   uniformly on $|x|\leq \frac{1}{2}\log\log\alpha.$
  	  Let $\ell_{\alpha}=\log(\sqrt{2\pi} \log(\alpha + e^{\frac1{\sqrt{2\pi}}})).$ Similar to the  Equations \eqref{e4} and \eqref{e3}, the following expressions for $v_\alpha(0,x)$ are given:
  	  	\begin{equation}
  	  	\aligned v_\alpha(0,x)&=\sqrt{\log (\alpha+e)}+\frac{x-\ell_{\alpha}}{\sqrt{\log (\alpha+e)}}+O((\log \alpha)^{-1/2}\alpha^{-1}) ;\\
  	  		  	v_\alpha^2(0, x)&=\log (\alpha+e)-2\ell_{\alpha}+2 x+\frac{(x-\ell_{\alpha})^{2}}{\log (\alpha+e)}+O(\alpha^{-1});
  	  	\\
  	  	\exp(-\frac{1}{2}v_\alpha^2(0, x))
  	  	&=\sqrt{\frac{2\pi }{\alpha}}\log (\alpha)\exp(-x-\frac{(\ell_{\alpha}-x)^{2}}{2\log(\alpha+e)})(1+O(\alpha_n^{-1})) \endaligned 
  	  \end{equation}
which implies $$ \beta_{\alpha}(x)=\frac{1+O((\log  \alpha)^{-1})}{(1+\frac{x-\ell_{\alpha}}{\log (\alpha+e)})^{2}}\exp(-x-\frac{(x-\ell_{\alpha})^2}{2\log  (\alpha+e)}).$$
It follows $$\aligned |\beta_\alpha(x)+e^{-x}|=\frac{e^{-x}(1+O((\log \alpha)^{-1}))}{(1+\frac{x-\ell_{\alpha}}{\log (\alpha+e)})^{2}}|\frac{(x-\ell_{\alpha})^2}{2\log(\alpha+e)}+\frac{2(x-\ell_{\alpha})}{\log(\alpha+e)}|=o(1), \endaligned $$
  	  uniformly on $|x|\leq \frac{1}{2}\log\log \alpha.$
  	  The same calculus in Theorem 1 for the case $\alpha=+\infty$ yields  $$\sup_{x\in \mathbb{R}}|\Phi_{\alpha}(x)-e^{-e^{-x}}|= \frac{(\log \log \alpha)^{2}}{2e\log \alpha}(1+o(1))\quad \text{for $\alpha\gg 1$}.$$  
  	  The proof is then completed. 
  
  \subsection{Proof of Remark \ref{upperbound}} For simplicity, use $v_{\alpha}$ to replace $v_{\alpha}(m, x)$ and rewrite $q_{1}(m, x)$ as 
  $$\aligned q_{1}(m, x)&=\frac{\sqrt{\alpha}}{6}v_{\alpha}^2-\frac{1}{2}\big(\sqrt{\alpha}(a+b x)-\frac12\big) v_{\alpha}+\frac{\sqrt{\alpha}}{2}(a+bx)^2-\frac12(a+bx)-\frac{\sqrt{\alpha}}{6}\\
  &=:A v_{\alpha}^2-B(x) v_{\alpha}+C(x) \endaligned $$   with $A=\frac{\sqrt{\alpha}}{6},$
  $$B(x)=\frac{1}{2}\big(\sqrt{\alpha}(a+b x)-\frac12\big) \quad \text{and} \quad C(x)=\frac{\sqrt{\alpha}}{2}(a+bx)^2-\frac12(a+bx)-\frac{\sqrt{\alpha}}{6}.$$
  Then 
  $$\aligned |\sum_{m=0}^{+\infty} q_{1}(m, x)\frac{\phi(v_{\alpha})}{\Phi(v_{\alpha})}|&\le A\sum_{m=0}^{+\infty} v_{\alpha}^2\frac{\phi(v_{\alpha})}{\Phi(v_{\alpha})}+|B(x)|\sum_{m=0}^{+\infty} |v_{\alpha}|\frac{\phi(v_{\alpha})}{\Phi(v_{\alpha})}+|C(x)|\sum_{m=0}^{+\infty} \frac{\phi(v_{\alpha})}{\Phi(v_{\alpha})}, \endaligned $$ 
  which in further ensures that   
 \begin{equation}\label{sumcontrol} \aligned \Phi_{\alpha}(x)\sum_{m=0}^{+\infty} \frac{|v^{k}_{\alpha}(m, x)|\phi(v_{\alpha}(m, x))}{\Phi(v_{\alpha}(m, x))}&\le \sum_{m=0}^{+\infty} |v_{\alpha}^k(m, x)|\phi(v_{\alpha}(m, x))\le \sqrt{\alpha},
  \endaligned \end{equation}
   where the last inequality holds because the second term is controlled by the following integral 
   $$\sqrt{\alpha}\int_{-\infty}^{+\infty} |t|^k \phi(t) dt$$ for $k=0, 1, 2.$ This  observation tells us that 
    \begin{equation}\label{sumtotal} \aligned &\quad \Phi_{\alpha}(x)\sum_{m=0}^{+\infty} |q_{1}(m, x)|\frac{\phi(v_{\alpha}(m, x))}{\Phi(v_{\alpha}(m, x))}\\
    &\le \frac{\alpha}{6}+(|B(x)|+|C(x)|)\Phi_{\alpha}(x)\sum_{m=0}^{+\infty} (|v_{\alpha}(m, x)|\vee 1)\frac{\phi(v_{\alpha}(m, x))}{\Phi(v_{\alpha}(m, x))}.
  \endaligned \end{equation} 
  As a direct consequence, for $x$ such that $|v_{\alpha}(0, x)|\le 1$ ensuring the boundedness of $|B(x)|+|C(x)|,$ we have  
   $$\aligned \sup_{x: \; |v_{\alpha}(0, x)|\le 1}\Phi_{\alpha}(x)|\sum_{m=0}^{+\infty} q_{1}(m, x)\frac{\phi(v_{\alpha})}{\Phi(v_{\alpha})}|
   &\le \frac{\alpha}{6}+\sqrt{\alpha}\sup_{x: \; |v_{\alpha}(0, x)|\le 1} (|B(x)|+|C(x)|)\\
   &\le \frac{4(\alpha+\sqrt{\alpha})}{3}.\endaligned $$ 
 For $x$ such that $v_{\alpha}(0, x)>1,$ it follows from the monotonicity of $t^k \phi(t)$ on $t>1$ that 
 \begin{equation}\label{sumkbound}\aligned \sum_{m=0}^{+\infty} |v_{\alpha}^k(m, x)|\phi(v_{\alpha}(m, x))&\le v_{\alpha}^k(0, x)\phi(v_{\alpha}(0, x))+\sqrt{\alpha}\int_{v_{\alpha}(0, x)}^{+\infty} t^k \phi(t) dt\\
 &\le v_{\alpha}^{k-1}(0, x)(v_{\alpha}(0, x)+\sqrt{\alpha})\phi(v_{\alpha}(0, x))\endaligned\end{equation}
 for $k=0, 1.$
 Therefore,  we have from \eqref{sumtotal} and \eqref{sumkbound} that 
 $$\aligned &\quad \Phi_{\alpha}(x)|\sum_{m=0}^{+\infty} q_{1}(m, x)\frac{\phi(v_{\alpha})}{\Phi(v_{\alpha})}|\le \frac{\alpha}{6}+\phi(v_{\alpha}(0, x))(v_{\alpha}(0, x)+\sqrt{\alpha})(|B(x)|+|C(x)|v_{\alpha}^{-1}(0, x)).\endaligned$$  
Hence, use the substitution $t=v_{\alpha}(0, x)$ to have 
$$\aligned &\quad \sup_{x: \; v_{\alpha}(0, x)>1}\Phi_{\alpha}(x)|\sum_{m=0}^{+\infty} q_{1}(m, x)\frac{\phi(v_{\alpha})}{\Phi(v_{\alpha})}|\\
 &\le \frac{\alpha}{6}+\sup_{t>1}\phi(t)(t+\sqrt{\alpha})(\sqrt{\alpha} t+\frac34+\frac{\sqrt{\alpha}}{6t})\\
 &\le \alpha+\frac{2\sqrt{\alpha}+1}{5}.\endaligned$$
 While for $x$ such that $v_{\alpha}(0, x)<-1,$ define 
  $$m_0:=\inf\{m: v_{\alpha}(m, x)\ge 0\},$$ which 
  implies with the monotonicity of $\Phi$ and $\Phi\le 1$ that $$\frac{\Phi_{\alpha}(x)}{\Phi(v_{\alpha}(m, x))}\le \Phi^{m_0-1}(v_{\alpha}(m_0-1, x))\le 2^{-m_0+1}.$$
   Thus, using the second inequality in \eqref{sumcontrol}, we have 
    $$\aligned  \Phi_{\alpha}(x)\sum_{m=0}^{+\infty} \frac{|v_{\alpha}(m, x)|^k\phi(v_{\alpha}(m, x))}{\Phi(v_{\alpha}(m, x))}&\le \sqrt{\alpha} 2^{-m_0+1}\leq\sqrt{\alpha} 2^{\sqrt{\alpha} v_{\alpha}(0, x)+1},\endaligned $$ 
 where we use the fact that $m_0\le -\sqrt{\alpha} v_{\alpha}(0, x).$
 Consequently, we have 
  $$\aligned \sup_{x: \; v_{\alpha}(0, x)<-1}\Phi_{\alpha}(x)|\sum_{m=0}^{+\infty} q_{1}(m, x)\frac{\phi(v_{\alpha})}{\Phi(v_{\alpha})}|&\le \frac{\alpha}{6}+2\sqrt{\alpha}\sup_{x: \; v_{\alpha}(0, x)<-1} 2^{\sqrt{\alpha} v_{\alpha}(0, x)}(|B(x)|+|C(x)|)\\
  &=\frac{\alpha}{6}+\sup_{t\ge \sqrt{\alpha}} 2^{-t}(t^2 +(1+\sqrt{\alpha})t+\frac{\sqrt{\alpha}}{2}+\frac{\alpha}{3})\\
  &\le \frac{\alpha}{6}+\frac{4}{3}.\endaligned$$ 
  When comparing these three suprema, we have  
   $$\aligned \sup_{x\in\mathbb{R}}\Phi_{\alpha}(x)|\sum_{m=0}^{+\infty} q_{1}(m, x)\frac{\phi(v_{\alpha})}{\Phi(v_{\alpha})}|&\le \frac{4}{3}(\alpha+\sqrt{\alpha}+1).\endaligned$$ 
   
Now we work on $\sup_{x\in\mathbb{R}}\Phi_{\alpha}(x)|\sum_{m=0}^{+\infty} q_{2}(m, x)\frac{\phi(v_{\alpha})}{\Phi(v_{\alpha})}|.$
A simple calculus indicates 
$$|q_2(m, x)|\le \frac{1}{2\alpha}|v_{\alpha}(m, x)|+(\frac1{2\alpha}-\frac{c_2}{b})|v_{\alpha}(0, x)|+c_1+\frac{c_2\alpha}{b}=:\frac{1}{2\alpha}|v_{\alpha}(m, x)|+D(x),$$
whence 
\begin{equation}\label{sumcontrol1}\Phi_{\alpha}(x)|\sum_{m=0}^{+\infty} q_{2}(m, x)\frac{\phi(v_{\alpha})}{\Phi(v_{\alpha})}|\le D(x)\Phi_{\alpha}(x)\sum_{m=0}^{+\infty}\frac{\phi(v_{\alpha})}{\Phi(v_{\alpha})}+\frac{1}{2\alpha}\Phi_{\alpha}(x)\sum_{m=0}^{+\infty} \frac{|v_{\alpha}|\phi(v_{\alpha})}{\Phi(v_{\alpha})}.\end{equation} Thereby, it follows from the second inequality of \eqref{sumcontrol} that 
 $$\aligned \sup_{x: \; |v_{\alpha}(0, x)|\le 1}\Phi_{\alpha}(x)|\sum_{m=0}^{+\infty} q_{2}(m, x)\frac{\phi(v_{\alpha})}{\Phi(v_{\alpha})}|
   &\le \sqrt{\alpha}\sup_{x: \; |v_{\alpha}(0, x)|\le 1} D(x)+\frac{1}{2\sqrt{\alpha}}\\
   &=\sqrt{\alpha}(\frac1{\alpha}+\frac{c_2(\alpha-1)}{b}+c_1).\endaligned $$ 
 and similarly 
   $$\aligned \sup_{x: \; v_{\alpha}(0, x)\le -1}\Phi_{\alpha}(x)|\sum_{m=0}^{+\infty} q_{2}(m, x)\frac{\phi(v_{\alpha})}{\Phi(v_{\alpha})}|
   &\le \frac{1}{2\sqrt{\alpha}}+2 \sqrt{\alpha} \sup_{x: \; v_{\alpha}(0, x)\le -1} 2^{\sqrt{\alpha}v_{\alpha}(0, x) } D(x)\\
   &\le \frac{1}{2\sqrt{\alpha}}+\frac{2}{e\ln 2}(c_1+\frac{c_2(\alpha-1)}{b}+\frac1{\alpha}).\endaligned $$  
   Here, for the last inequality we use $\sup_{t>0} 2^{-t} t=(e\ln 2)^{-1}.$

 We also have from \eqref{sumkbound} that 
   $$\aligned &\quad \sup_{x: \; v_{\alpha}(0, x)\ge 1}\Phi_{\alpha}(x)|\sum_{m=0}^{+\infty} q_{2}(m, x)\frac{\phi(v_{\alpha})}{\Phi(v_{\alpha})}|\\
 &\le \sup_{x: \; v_{\alpha}(0, x)\ge 1} \phi(v_{\alpha}(0, x))(v_{\alpha}(0, x)+\sqrt{\alpha} )(v_{\alpha}(0, x)^{-1}D(x)+\frac{1}{2\alpha})\\
 &=\sup_{t>1} \phi(t)\big((\frac{1}{\alpha}-\frac{c_2}{b}) t+\frac{c_1b+c_2\alpha}{b}+(\frac{1}{\alpha}-\frac{c_2}{b})\sqrt{\alpha}+\frac{c_1b+c_2\alpha}{b t} \sqrt{\alpha}\big)\\
 &=e^{-\frac12}(\sqrt{\alpha}+1)(\frac{1}{\alpha}+c_1+\frac{c_2(\alpha-1)}{b}).
 \endaligned$$ 
 Comparing these three upper bounds, we have 
$$\aligned \sup_{x\in\mathbb{R}}\Phi_{\alpha}(x)|\sum_{m=0}^{+\infty} q_{2}(m, x)\frac{\phi(v_{\alpha})}{\Phi(v_{\alpha})}|&\le \frac{2}{e\ln 2}(c_1+\frac{c_2(\alpha-1)}{b}+\frac1{\alpha})(1+\sqrt{\alpha}).\endaligned$$

\subsection*{Acknowledgment}   Yutao Ma gratefully acknowledges partial support from the National Natural Science Foundation of China (Grants No. 12171038 and 12571149) and the 985 Project. The authors also extend their sincere thanks to Professor Forrester for bringing to their attention several valuable references.

\end{document}